%% file: main.tex
\documentclass[letterpaper, 11pt]{article}
\usepackage{graphicx} 

\title{A Block Coordinate and Variance-Reduced Method\\ for Generalized Variational Inequalities of Minty Type}
\author{Jelena Diakonikolas\footnote{
Department of Computer Sciences, 
University of Wisconsin-Madison, 
\texttt{jelena@cs.wisc.edu}}}
\date{}

\input{prologue}

\begin{document}

\maketitle

\begin{abstract}
    Block coordinate methods have been extensively studied for minimization problems, where they come with significant complexity improvements whenever the considered problems are compatible with block decomposition and, moreover, block Lipschitz parameters are highly nonuniform. For the more general class of variational inequalities with monotone operators, essentially none of the existing methods transparently shows potential complexity benefits of using block coordinate updates in such settings. Motivated by this gap, we develop a new randomized block coordinate method and study its oracle complexity and runtime. We prove that in the setting where block Lipschitz parameters are highly nonuniform---the main setting in which block coordinate methods lead to high complexity improvements in any of the previously studied settings---our method can lead to complexity improvements by a factor order-$m$, where $m$ is the number of coordinate blocks. The same method further applies to the more general problem with a finite-sum operator with $m$ components, where it can be interpreted as performing variance reduction. Compared to the state of the art, the method leads to complexity improvements up to a factor $\sqrt{m},$ obtained when the component Lipschitz parameters are highly nonuniform. 
\end{abstract}

\section{Introduction}

Block coordinate methods---which update only a subset of variables at a time---are a fundamental class of optimization methods with a long history in optimization research and practice. Some of the earliest examples falling in this category are the method of Kaczmarz from 1937 \cite{karczmarz1937angenaherte} for solving linear systems of equations and Osborne's algorithm from 1960 \cite{osborne1960pre} for matrix balancing. On the theoretical front, research on block coordinate methods flourished following the introduction of randomized block coordinate methods and techniques for analyzing them \cite{strohmer2009randomized,nesterov2012efficiency,leventhal2010randomized}, in the context of convex minimization. 

Broadly speaking, there are two major classes of block coordinate methods that are used as a black box in practice and have been the subject of research in optimization algorithms, categorized according to the order in which blocks of coordinates are updated: (1) randomized methods \cite{nesterov2012efficiency,strohmer2009randomized,nesterov2017efficiency,qu2016coordinate}, which fix a probability distribution over the coordinate blocks and choose which block to update by sampling with replacement; and (2) cyclic methods \cite{beck2013convergence,song2023cyclic,ortega1970iterative,sun2019worst}, which select a (deterministic or random) permutation of blocks and update all coordinate blocks in one pass, in order following the selected permutation. 

At present, convergence properties of randomized methods  are well-understood in convex optimization and in primal-dual methods for min-max optimization, where some versions can also be seen as performing variance reduction (see, e.g., \cite{song2021variance,mehta2024primal,alacaoglu2022complexity,carmon2020coordinate}). In the setting of convex optimization, there is a clear separation between randomized methods, cyclic methods, and traditional full vector update (i.e., single block) methods \cite{sun2019worst}. In particular, for problems where block coordinate updates are meaningful (i.e., where their update cost is proportional to the block size), the complexity of randomized methods is never higher than the complexity of full vector update methods, while it can be lower by a factor scaling linearly or with the square-root of the total number of blocks $m$ (often as large as the problem dimension) \cite{nesterov2012efficiency,nesterov2017efficiency,lee2013efficient}. By contrast, the complexity of cyclic methods is usually higher than the complexity of full vector update methods \emph{in the worst case}, by factors scaling polynomially (as $m^\alpha,$ for $\alpha \in\{1, 1/2, 1/4\}$) with the number of blocks \cite{song2023cyclic,beck2013convergence,sun2019worst,lin2014accelerated}.\footnote{It is of note that despite their unfavorable worst-case complexity, cyclic methods are often preferred in practice, which has motivated recent research on obtaining more fine-grained complexity bounds for this class of methods; see, e.g., \cite{wright2020analyzing,lee2019random,gurbuzbalaban2017cyclic,song2023cyclic}.}

On the other hand, much less is understood about convergence of block coordinate methods for solving the more general class of  variational inequality (VI) problems. VIs associated with monotone operators are equilibria problems that strictly generalize convex minimization and convex-concave min-max optimization. In Euclidean settings, they are further closely related to solving fixed-point equations with nonexpansive (1-Lipschitz) operators. For these general problems, current results on block coordinate methods in the literature \cite{yousefian2018stochastic,he2020point,kotsalis2022simple,song2023cyclic} do not even transparently show that block coordinate updates can lower the algorithm complexity under standard assumptions defining the associated problem classes that these methods address.

The main contribution of this work is a new randomized block coordinate method for a general class of variational inequalities (defined below, in the next subsection). Our method has provably lower complexity than full vector update methods under a block Lipschitz assumption, in the setting where block Lipschitz constants are highly nonuniform---the main setting in which block coordinate methods are known to have an edge over full vector update methods. The potential complexity improvements in this regime are by a factor scaling linearly with the number of blocks $m.$ 

The same method can be used more generally to address problems where the operator has finite-sum structure with $m$ components, in which case our method can be interpreted as performing variance reduction. Compared to state of the art variance-reduced methods \cite{alacaoglu2022stochastic,cai2024variance,carmon2019variance}, there are additional potential complexity gains of order-$\sqrt{m}$ when the component Lipschitz parameters are highly nonuniform. 

To make these statements more specific, we next formally introduce the considered problem setup and state our main results, followed by an overview of related work. 

\subsection{Problem Setup}\label{sec:problem-setup}

We are interested in solving Generalized (or regularized) Minty Variational Inequality (GMVI) problems, in which the goal is to find $\vx^* \in \dom(g)$ such that
\begin{equation}\label{eq:main-problem}\tag{P} 
\innp{\mF(\vx), \vx - \vx^*} + g(\vx) - g(\vx^*) \geq 0, \; \forall \vx \in \sR^d, 
\end{equation}
where $\mF:\sR^d \to \sR^d$ is a monotone Lipschitz operator and $g: \sR^d \to \sR\cup \{+\infty\}$ is a proper, extended-valued, convex, lower semicontinuous 
function with an efficiently computable proximal operator. 
Given an error parameter $\epsilon > 0,$ the $\epsilon$-approximation variant of this problem is to find $\vx^*_\epsilon$ such that
\begin{equation}\label{eq:main-problem-eps}\tag{P$_\epsilon$} 
\innp{\mF(\vx), \vx^*_\epsilon - \vx} - g(\vx) + g(\vx^*_\epsilon) \leq \epsilon, \; \forall \vx \in \sR^d. 
\end{equation}
For conciseness, we further define
\begin{equation}\label{eq:gap-def}
    \Gap(\vx^*_\epsilon, \vx) = \innp{\mF(\vx), \vx^*_\epsilon - \vx} - g(\vx) + g(\vx^*_\epsilon)
\end{equation}
for arbitrary but fixed $\vx \in \R^d,$ and equivalently focus on finding $\vx^*_\epsilon$ such that $\Gap(\vx^*_\epsilon, \vx) \leq \epsilon.$ 
(Here we assume that $g$ is either strongly convex or its domain is compact, as otherwise such guarantees are not possible. An alternative is that the above inequality holds for all $\vx$ from a compact set, which could be, for example, a ball of constant radius, centered at $\vx^*,$ folowing similar considerations as in \cite{nesterov2007dual}.) 

To simultaneously handle the cases in which $\mF$ is possibly strongly monotone and/or $g$ is strongly convex, we assume that $\mF$ is only monotone and $g$ is strongly convex with parameter $\gamma \geq 0$ (when $\gamma = 0,$ it is simply convex). This is without loss of generality.

Throughout the paper, we assume that we are given a finite-sum decomposition of $\mF,$ meaning that there are operators $\mF_1, \dots, \mF_m$ such that $\mF$ can be expressed as
\begin{equation}\label{eq:finite-sum}
    \mF(\vx)  = \sum_{j=1}^m \mF_j(\vx), \quad \forall \vx \in \dom(g).
\end{equation}

A prototypical example that motivated this research and we use throughout to illustrate our results comes from block decomposition of the coordinates of $\vx,$ in which case our method can be viewed as a block coordinate method. 
For block coordinate strategies to make sense computationally, when  considering block coordinate settings, we assume that $F$ is ``block coordinate friendly,'' meaning that there exists a partition of the set $\{1, 2, \dots, d\}$ into $m \leq d$ subsets $\cS^1, \cS^2, \dots, \cS^m$ such that the computational complexity of evaluating $\mF\bl{1}(\vx_1), \mF\bl{2}(\vx_2), \dots \mF\bl{m}(\vx_m)$ at $m$ different points $\vx_1, \dots, \vx_m \in \dom(g)$ is of the same order as the complexity of evaluating the full operator $\mF(\vx),$ at any point $\vx$, where $\mF\bl{j}(\vx)$ for $j \in \{1, \dots, m\}$ denotes the coordinates of $\mF(\vx)$ indexed by the elements of $\cS^j.$ This is true, for example, for operators that are linear, block separable, or expressible as a sum of operators that are block coordinate friendly with respect to the same block partition. Function $g(\vx) = \sum_{j=1}^m g^j(\vx\bl{j})$ in this setting is assumed to be block separable w.r.t.\ the same block partition under which $\mF$ is block coordinate friendly.

As is standard, we further define linear operators $\mU^j,$ $j \in \{1,\dots, m\}$, to be such that for any $j \in \{1,\dots, m\}$ and any $\vx,$ we have that the vector $\mU^j \vx$ is such that $(\mU^j \vx)\bl{j} = \vx\bl{j}$, while all its remaining elements are zero. It should then become clear that in the block coordinate setting, we can use $\mF_j = \mU^j \mF$ to write our problem in the finite-sum form stated in \eqref{eq:finite-sum}. 

For the purpose of streamlined comparison to the literature, we will sometimes assume that each component $\mF_j$ is $L_j$-Lipschitz continuous for some parameter $L_j \in (0, \infty),$ $j \in \{1, \dots, m\}.$ In particular, in such cases we assume that 
for any $\vx, \vy \in \dom(g)$ and any $j \in \{1, \dots, m\},$
\begin{equation}\label{eq:block-Lip-assmpt}
    \|\mF_j(\vx) - \mF_j(\vy)\|_{*} \leq L_j \|\vx - \vy\|,  
\end{equation}
where $\|\cdot\|_*$ and $\|\cdot\|$ are a pair of norms that are dual to each other. In the block coordinate settings, under a mild assumption that $\|\mF\bl{j}(\vx)\|_* = \|\mU^j\mF(\vx)\|_*,$ for all $\vx \in \dom(g),$ this is the same assumption as made in prior work addressing the same class of problems and methods \cite{kotsalis2022simple}. 
We take $L_1, L_2, \dots, L_m$ to be the smallest  constants for which \eqref{eq:block-Lip-assmpt} holds for all $\vx, \vy \in \dom(g)$ and denote $\vlambda := (L_1, L_2, \dots, L_m)^\top$. 

If $L$ is the Lipschitz constant of the full operator $\mF,$ i.e., if for all $\vx, \vy \in \dom(g)$,
\begin{equation}\label{eq:full-Lip-assmpt}
    \|\mF(\vx) - \mF(\vy)\|_* \leq L \|\vx - \vy\|,  
\end{equation}
and $L$ is chosen as the smallest such constant, then  
we have 
\begin{equation}\label{eq:component-vs-full-Lip}
    \|\vlambda\|_\infty = \max_{1\leq i \leq m} L_i \leq L \leq \sum_{i=1}^m L_i = \|\vlambda\|_1.
\end{equation}  
Both inequalities in \eqref{eq:component-vs-full-Lip} are generally tight, in the sense that for either inequality there exists an operator that satisfies it with equality. However, it is also the case that either (or both) of these inequalities can be loose on general problem instances.  We are particularly interested in the regime where $\vlambda$ is highly non-uniform, which results in $\|\vlambda\|_\infty$ and $\|\vlambda\|_1$ being of the same order. This is the primary regime in which block coordinate methods are used and lead to improved complexity bounds and practical speedups compared to full vector update (single-block) methods, in any of the optimization settings mentioned in the introduction. 

Our results are expressed in terms of a parameter $\Lpq$ that depends on probability distributions $\vp, \vq \in \Delta_m$ used in our algorithm, defined by
\begin{equation}\label{eq:Lpq-def}
    \Lpq := \sqrt{\sup_{\vx, \vy \in \dom(g), \vx \neq \vy} \frac{\sum_{j=1}^m \frac{1}{p_j {q_j}^2}\|\mF_j(\vx) - \mF_j(\vy)\|_{*}^2}{\|\vx - \vy\|^2}},
\end{equation}
as this bound can sometimes provide a smaller constant than what we get from the worst-case block coordinate assumptions in \eqref{eq:block-Lip-assmpt} (see Section \ref{sec:apps} for  examples). Observe that under \eqref{eq:block-Lip-assmpt}, $\Lpq \leq \sqrt{\sum_{j=1}^m \frac{1}{p_j {q_j}^2} {L_j}^2}.$ 

\subsection{Main Result}\label{sec:main-results}

Our main result is an algorithm that outputs a solution with error $\epsilon > 0$ (where the error is defined as the expected gap in the general regime or distance to the solution $\vx_*$ in the strongly monotone regime) after 
\begin{equation}\notag
     \cO\bigg(\min\bigg\{\frac{\Lpq \sup_{\vx \in \dom(g)} D(\vx, \vx_0)}{\epsilon},\; \Big(\max_{1\leq j \leq m} \frac{1}{q_j} + \frac{\Lpq}{\gamma}\Big)\log\Big(\frac{\Lpq D(\vx_*, \vx_0)}{\epsilon}\Big) \bigg\}\bigg).
\end{equation}
oracle queries to blocks $\mF\bl{j}$ of coordinates of $\mF$ in block coordinate settings, or component operators $\mF_j$ in the more general finite-sum settings. In particular, under block/component Lipschitz assumption \eqref{eq:block-Lip-assmpt} and importance-based sampling probabilities $\vp, \vq$ (see Section \ref{sec:sampling-distributions} for precise definitions), this complexity is
\begin{equation}\label{eq:it-compl-main}
  \cO\bigg(\min\bigg\{\frac{\|\vlambda\|_{1/2} \sup_{\vx \in \dom(g)} D(\vx, \vx_0)}{\epsilon},\; \Big(m + \frac{\|\vlambda\|_{1/2}}{\gamma}\Big)\log\Big(\frac{\|\vlambda\|_{1/2} D(\vx_*, \vx_0)}{\epsilon}\Big) \bigg\}\bigg),
\end{equation}
where $\|\vlambda\|_{1/2} = \big(\sum_{j=1}^m {L_j}^{1/2}\big)^2$.

We recall here that the optimal oracle complexity measured in terms of the oracle queries to the full operator $\mF$ and attained by methods such as Mirror-Prox/Extragradient \cite{nemirovski2004prox,korpelevich1976extragradient}, Nesterov's dual extrapolation \cite{nesterov2007dual}, Popov's method \cite{popov1980modification}, and their variants, is
\begin{equation}\label{eq:it-compl-full-update}
  \cO\bigg(\min\bigg\{\frac{L \sup_{\vx \in \dom(g)} D(\vx, \vx_0)}{\epsilon},\; \frac{L}{\gamma}\log\Big(\frac{L D(\vx_*, \vx_0)}{\epsilon}\Big) \bigg\}\bigg),
\end{equation}
where $L$ is the Lipschitz parameter of $\mF$ and, as noted above, is at least $\|\vlambda\|_\infty$ but can be as large as $\|\vlambda\|_1.$

In standard applications of block coordinate methods, the cost of evaluating one block $\mF\bl{j}$ of $\mF$ is of the order of $1/m$ of the cost of evaluating the full operator $\mF,$ on average. Thus the total number of full operator evaluations in this setting becomes $1/m$ times the bound in \eqref{eq:it-compl-main}. Because $\|\vlambda\|_{\infty} \leq \|\vlambda\|_{1/2} \leq m^2 \|\vlambda\|_{\infty}$, in the worst case where $\vlambda$ is (close to) uniform and $L = \|\vlambda\|_\infty,$ the complexity of our method can be $m$ times worse than the complexity of full vector update methods. However, this is also the setting in which block coordinate methods on other problem classes generally do not improve over full vector update methods. On the other end of the spectrum, where $\vlambda$ is highly non-uniform---the primary setting in which block coordinate methods are  used---we have that $\|\vlambda\|_\infty$ and $\|\vlambda\|_{1/2}$ are of the same order, in which case our method's complexity is no worse than the complexity of full vector update methods and is $m$ times lower if $L/\gamma \geq m$. Similar conclusions apply when comparing our method to methods not accounting for sum-decomposability \cite{nemirovski2004prox,korpelevich1976extragradient,nesterov2007dual,popov1980modification} in the more general finite-sum settings \eqref{eq:finite-sum}. 

\subsection{Related Work}

As mentioned before, block coordinate methods have been broadly investigated as pertaining to (convex or nonconvex) minimization problems \cite{qu2016coordinate,nesterov2012efficiency,nesterov2017efficiency,diakonikolas2018alternating,beck2013convergence,lin2014accelerated,cai2023cyclic,lee2013efficient,allen2016even,richtarik2014iteration,fercoq2015accelerated,tseng2009coordinate} and a specific subclass of GMVI corresponding to min-max optimization, focusing on primal-dual methods, where coordinate updates are typically performed either only on the primal or only on the dual side \cite{alacaoglu2022complexity,song2021variance,chambolle2018stochastic,mehta2024primal,carmon2020coordinate,dang2014randomized,alacaoglu2017smooth,carmon2019variance,latafat2019new, fercoq2019coordinate, alacaoglu2020random}.

When it comes to the general GMVI class considered in this work, much less is understood about the convergence of block coordinate methods. Initial results on randomized block coordinate methods \cite{yousefian2018stochastic} considered milder block Lipschitz assumptions than \eqref{eq:block-Lip-assmpt}, stated for vectors $\vx, \vy$ differing only over the coordinate block $j$---more in line with the literature on minimization problems---but in turn required additional assumptions such as bounded norm of block operators $\mF\bl{j}$ and led to slower convergence rates of the order $1/\sqrt{k}$ for general monotone operators and a sublinear $1/k$ rate under strong monotonicity. 
Similarly, \cite{chow2017cyclic} obtained results for a cyclic block coordinate method with the same $1/\sqrt{k}$ convergence rate, under the same block Lipschitz assumption as \eqref{eq:block-Lip-assmpt}, assuming $\mF$ is cocoercive. On the other hand, results with the optimal sublinear convergence rate $1/k$ in the general case and linear convergence rate in the case of strongly convex $g$ are very recent and  are all applicable under block Lipschitz assumptions bounding the change in the operator $\mF\bl{j}$ between any two vectors $\vx, \vy,$ similar to \eqref{eq:block-Lip-assmpt}. The two main results in this domain are \cite{song2023cyclic} and \cite{kotsalis2022simple}, addressing cyclic and randomized block coordinate methods, respectively. 

The work concerning cyclic block coordinate methods \cite{song2023cyclic} provides the first cyclic method for GMVI with the optimal $1/k$ convergence rate. The block Lipschitz assumption is somewhat different than ours, defined by $\|\mF\bl{j}(\vx) - \mF\bl{j}(\vy)\|_2 \leq \sqrt{(\vx - \vy)^\top \mQ^j (\vx - \vy)}$ for some fixed symmetric positive semidefinite matrices $\mQ^j.$ Their complexity result is expressed in terms of a summary Lipschitz constant $\hat{L} = \sqrt{\|\sum_{j=1}^m \widehat{\mQ}^j\|_2}$, where $\|\cdot\|_2$ denotes the operator/spectral norm of matrices, and $\widehat{\mQ}^j$ are matrices obtained from $\mQ^j$ by zeroing out rows and columns corresponding to the coordinates from the first $j-1$ coordinate blocks. This is a fine-grained complexity result that is not directly comparable to ours in full generality. When specialized to the block Lipschitz assumption \eqref{eq:block-Lip-assmpt}, which would correspond to choosing $\mQ^j = {L_j}^2 \mI$, where $\mI$ is the identity matrix, we have that $\hat{L} = \|\vlambda\|_2$ and the resulting complexity (number of full cycles) in \cite{song2023cyclic} boils down to $\cO\Big(\min\Big\{\frac{\|\vlambda\|_2\sup_{\vx \in \dom(g) }D(\vx, \vx_0)}{\epsilon}, \, \frac{\|\vlambda\|_2}{\gamma}\log\Big(\frac{L D(\vx_*, \vx_0)}{\epsilon}\Big)\Big\}\Big)$. Because the results from \cite{song2023cyclic} are for Euclidean setups and in this case it can be argued that $\|\vlambda\|_2/\sqrt{m} \leq L \leq \|\vlambda\|_2$, where we recall $L$ to be the Lipschitz constant of the full operator $\mF,$ we get that the complexity from \cite{song2023cyclic} in this setting is never better than the complexity of full vector update methods (stated in \eqref{eq:it-compl-full-update}), while it is potentially worse by a factor $\sqrt{m}.$ By comparison, as discussed in Section \ref{sec:main-results}, our complexity can be worse than that of full-vector update methods by a factor $m$, while on the other hand it can also be better by a factor $m,$ which happens when $\vlambda$ is highly nonuniform.  

In the domain of randomized block coordinate methods for GMVI, \cite{kotsalis2022simple} is the only result we are aware of that attains the optimal convergence rate in the general monotone case. Compared to our result stated in \eqref{eq:it-compl-main}, $\|\vlambda\|_{1/2}$ is replaced by $m^2 \|\vlambda\|_\infty,$ while the dependence on the remaining problem parameters is the same. Since $\|\vlambda\|_{1/2} \leq m^2 \|\vlambda\|_\infty$, our result strictly improves over \cite{kotsalis2022simple}, with the improvement being of the order $m^2$ when $\vlambda$ is highly nonuniform (i.e., when $\|\vlambda\|_{1/2} \approx \|\vlambda\|_\infty$). It is worth noting that \cite{kotsalis2022simple} also considers other problems involving stochastic variational inequalities, motivated by applications in reinforcement learning, which are outside  the scope of this work.

Finally, our work is related to the literature on variance reduction algorithms for finite-sum problems, a topic intensely investigated in the machine learning literature in the past decade \cite{johnson2013accelerating, defazio2014saga, nguyen2017sarah,schmidt2017minimizing, allen2017katyusha, palaniappan2016stochastic,chavdarova2019reducing,carmon2019variance,alacaoglu2022stochastic,gower2020variance,song2021variance,woodworth2016tight,han2021lower}. We compare to this literature separately from block coordinate methods, as their focus is different (improving the dependence on the number of components in the finite-sum), while there are typically no considerations regarding possibly sparse updates when component operators $\mF_j$ are sparse. The main examples of stochastic variance-reduced estimators include (1) estimators that are based on infrequently computing high-accuracy (e.g., full $\mF$) estimates of $\mF$ and combining them with low-accuracy (e.g., single randomly sampled $\mF_j$ and appropriately rescaled) estimates of $\mF_j$ computed in the remaining iterations. These include estimators such as SVRG \cite{johnson2013accelerating,kovalev2020don} and SARAH/SPIDER/PAGE \cite{nguyen2017sarah,fang2018spider,li2021page}; (2) estimators that are based on computing $\mF$ only once, at initialization, and using low-complexity and low-accuracy estimates in the entire algorithm run, at the expense of potentially higher memory requirements. These are known as SAGA-type estimators \cite{defazio2014saga,schmidt2017minimizing}. Our variance-reduced method belongs to the latter category. 

In the context of finite-sum GMVI, there is a series of (mostly recent) results \cite{palaniappan2016stochastic,carmon2019variance,alacaoglu2022stochastic,cai2024variance} that are optimal in the sense of oracle complexity under a suitable ``Lipschitz in expectation'' assumption about the operator $\mF$ \cite{han2021lower}. Under component Lipschitz assumptions \eqref{eq:block-Lip-assmpt} it is not known whether these methods are optimal or off by a factor $\sqrt{m};$ see the discussion in \cite[Section 5.6]{han2021lower}. Because of the different Lipschitz assumptions on $\mF$ and its components, the results from this prior work are not directly comparable to ours in full generality. If one takes the most natural and usual assumption that individual estimators for $\mF$ are chosen as $\mF_j/p_j$ sampled with probability $p_j$ according to some probability distribution $\vp = (p_1, \dots, p_m)^\top,$ then the complexity results in state of the art methods \cite{cai2024variance,alacaoglu2022stochastic,carmon2019variance} would replace $\Lpq$ in our bound \eqref{eq:it-complexity-general} by $\sqrt{m} M,$ where  
\begin{equation}\notag
     M := \sqrt{\sup_{\vx, \vy \in \dom(g), \vx \neq \vy} \frac{\sum_{j=1}^m \frac{1}{p_j}\|\mF_j(\vx) - \mF_j(\vy)\|_{*}^2}{\|\vx - \vy\|^2}}.
\end{equation}
When specialized to component Lipschitzness in \eqref{eq:block-Lip-assmpt}, the difference between the complexity bounds is in scaling with $\|\vlambda\|_{1/2}$ (in our work) versus $\sqrt{m}\|\vlambda\|_2$ (in prior work). Thus, our bound can be worse by a factor up to $\sqrt{m}$, which occurs when $\vlambda$ is (near-)uniform (in which case variance-reduced methods have worse complexity than classical methods such as \cite{nemirovski2004prox,korpelevich1976extragradient,nesterov2007dual,popov1980modification}). Notably, our bound is better by a factor  $\sqrt{m}$ than prior work when $\vlambda$ is highly nonuniform, which is the main setting in which variance-reduced methods exhibit  significant improvements over classical methods.  

Further comparison to related work on specific example problems is provided in Section \ref{sec:apps}. 

\section{Randomized Extrapolated Method for GMVI: Algorithm Description}\label{sec:algo}

We begin our technical discussion by describing our algorithm, for which the complete pseudocode is provided in Algorithm \ref{alg:main}. We state it in a form that is convenient for the analysis but that enforces full vector updates for $\vx_k$. However, as we discuss below, the algorithm can be modified to perform lazy updates and ensure only a subset of the coordinates is updated in each iteration, ensuring low per-iteration cost in the setting of block coordinate-style methods.

\begin{algorithm}
\caption{Randomized Extrapolated Method (REM; Analysis Version)}\label{alg:main}
    \begin{algorithmic}[1]
        \State \textbf{Input}: $\vx_0 \in \dom(g),\, \vp \in \Delta_m,\, \vq \in \Delta_m, K, \gamma$
        \State \textbf{Initialization}: $\mFt_{0, j} \leftarrow \mF_j(\vx_0)$ for $j \in \{1, \dots, m\},$ $\mFt_0 \leftarrow \mF(\vx_0),$ $a_0 \leftarrow A_0 \leftarrow 0,$ $\vz_0 \leftarrow 0$
        \For{$k = 1:K$}
        \State Choose step size $a_k$ (see Theorem \ref{thm:main}) and update  $A_k \leftarrow a_k + A_{k-1}$
        \State \parbox[t]{\dimexpr\textwidth-\leftmargin-\labelsep-\labelwidth}{Randomly draw $j_k$ from $\{1, \dots, m\}$ according to the probability distribution $\vp \in \Delta_m,$ independently of any prior random algorithm choices\strut}
        \State $\mFh_k \leftarrow \mFt_{k-1} + \frac{a_{k-1}}{a_k p_{j_k}}(\mF_{j_k}(\vx_{k-1}) - \mFt_{k-2, j_k})$
        \State $\vz_k \leftarrow \vz_{k-1} + a_k \mFh_k$
        \State $\vx_k \leftarrow \argmin_{\vu \in \R^d}\Big\{\innp{\vz_k, \vu} + A_k g(\vu) + D(\vu, \vx_{0})\Big\}$
        \State \parbox[t]{\dimexpr\textwidth-\leftmargin-\labelsep-\labelwidth}{Randomly draw $j_k'$ from $\{1, \dots, m\}$ according to the probability distribution $\vq \in \Delta_m,$ independently of any prior random algorithm choices\strut}
        \State $\mFt_{k, j} \leftarrow \begin{cases} \mFt_{k-1, j}, & \text{ if } j \neq j_k'\\ \mF_{j}(\vx_k), & \text{ if } j = j_k' \end{cases}$
        \EndFor
    \end{algorithmic}
    \Return $\vxb_k = \frac{1}{A_k}\sum_{i=1}^k a_i \vx_i$ or $\vx_k$
\end{algorithm}

We now discuss how the algorithm is derived. 
Algorithm \ref{alg:main} is an iterative algorithm that follows dual-averaging- (or ``lazy'' mirror descent-)style updates defined by 
\begin{equation}\label{eq:alg-iteration}
    \vx_{k} = \argmin_{\vu \in \R^d}\Big\{\sum_{i=1}^k a_i \big(\innp{\mFh_i, \vu} + g(\vu)\big) + D(\vu, \vx_{0})\Big\},
\end{equation}
where $\mFh_i$ are conveniently chosen ``extrapolated operators'' (to be defined shortly), $a_i$ are algorithm step sizes that determine the convergence rate, and $D(\vu, \vx_{0})$ is the Bregman divergence of a function that is 1-stronly convex w.r.t.\ $\|\cdot\|$ and chosen so that the minimization problem in \eqref{eq:alg-iteration} is efficiently solvable. When REM (Algorithm \ref{alg:main}) is specialized to block coordinate settings, $D(\vx, \vx_0)$ is further assumed to be block separable so that $D(\vu, \vx_{0}) = \sum_{j=1}^m D^j(\vu\bl{j}, \vx_{0}\bl{j})$. 

Ignoring the Bregman divergence term, which is problem-specific (and there are standard choices depending on the norm $\|\cdot\|$ and function $g$), to fully specify the algorithm iterations, we need to define $\mFh_i$ and the step sizes $a_i.$ The latter will be set by the analysis, while for the former we take what can be interpreted as an \emph{extrapolated operator} update, defined by $\mFh_k \leftarrow \mFt_{k-1} + \frac{a_{k-1}}{a_k p_{j_k}}(\mF_{j_k}(\vx_{k-1}) - \mFt_{k-2, j_k})$. 
Similar operator extrapolation strategies appear in prior work on block coordinate methods for GMVI \cite{song2023cyclic,kotsalis2022simple}, though the specific choices are different in both algorithms from prior work and our algorithm (e.g., only our work combines the full estimator $\mFt_{k-1}$ with partial extrapolation $\mF_{j_k}(\vx_{k-1}) - \mFt_{k-2, j_k}$ in the update for $\mFh_k$ and uses rescaling by the sampling probability $p_{j_k}$). For the operator estimator $\mFt_k = \sum_{j=1}^m \mFt_{k, j},$ 
we maintain a table with entries $\mFt_{i, j}$, of which only the entry $j_i$ is updated in iteration $i,$ as stated in Line 10, Algorithm~\ref{alg:main}.

\paragraph{Block coordinate considerations for Algorithm \ref{alg:main}.} 
In the setting of block coordinate methods, we can define $\mFh_k$ 
via  
\begin{equation}\label{eq:F-hat-def}
    \mFh_k = \mFt_{k-1} + \frac{a_{k-1}}{a_k p_{j_k}}\mU^{j_k}(\mF(\vx_{k-1}) - \mFt_{k-2}), 
\end{equation}
where $\mFt_k$ is a list of block operators $\mF\bl{1}, \mF\bl{2}, \dots, \mF\bl{m},$ updated in a block coordinate manner. In particular, in this setting, $\mFt_k$ can be initialized as $\mFt_0 = \mF(\vu_0)$, while in each iteration $k \geq 1$, a randomly selected block $j_k' \in \{1, \dots, m\}$ is updated by setting $\mFt_k\bl{j_k'} = \mF\bl{j_k'}(\vx_{k-1}).$

Although this definition technically leads to a full vector update for $\vx_k$ (as $\mFh$ can have all non-zero entries), we observe that the blocks of $\vx_k$ can be  updated lazily, resulting in the same per-iteration complexity as the more traditional block coordinate methods. In particular, let $A_k = \sum_{i=1}^k a_i$ and $\vz_k = \sum_{i=1}^k a_i \mFh_i,$ so that $\vx_k$ can be equivalently defined via 
\begin{equation}\label{xk-equiv-block-def}
    \vx_k\bl{j} = \argmin_{\vu\bl{j} \in \R^{|\cS^j|}}\Big\{\innp{\vz_k\bl{j}, \vu\bl{j}} + A_k g^j(\vu\bl{j}) + D^j(\vu\bl{j}, \vx_0\bl{j})\Big\}.
\end{equation}
Although this update is indeed a full vector update, it does not need to be carried out in every iteration. Instead, (blocks of) coordinates can be updated lazily, when they are needed: either because they got selected (i.e., when $j = j_k$) or if they are used for the computation of $F\bl{i_k}(\vx_{k-1})$. Here we are tacitly assuming that reading a block of coordinates costs as much as updating them: this is true in most cases of interest, particularly under the assumption that the update \eqref{xk-equiv-block-def} is computable in closed form (which is the main setting in which the considered class of methods is used). The key observation is that at any iteration, vector $\vz_k\bl{j}$ for $j \neq j_k$ can succinctly be expressed as
\begin{equation}
    \vz_k\bl{j} = (A_k - A_{t_{k, j}-1})\mFt_k\bl{j}. 
\end{equation}
where $t_{k, j}$ denotes the last iteration up to $k$ in which $j$ was selected by the random sampling procedure. 
A lazy, implementation version of Algorithm \ref{alg:main} is provided in Appendix \ref{appx:lazy-algo}.

\section{Convergence Analysis} 

To analyze the algorithm, we directly bound $\Gap(\vxb_k, \vx)$ for an arbitrary but fixed $\vx \in \dom(g),$ where $\vxb_k := \frac{1}{A_k}\sum_{i=1}^k a_i \vx_i$. The argument is constructive and at a high level follows the line of work \cite{diakonikolas2019approximate,diakonikolas2021efficient,diakonikolas2018accelerated,cohen2018acceleration,diakonikolas2018alternating,diakonikolas2020locally,diakonikolas2020fair,diakonikolas2024complementary,chakrabarti2023block,lin2023accelerated} originating with the approximate duality gap technique framework \cite{diakonikolas2019approximate}, which is in turn closely related to Nesterov's estimating sequences \cite{nesterov2005smooth}. The basic idea is to create an upper estimate on the notion of a gap (which for us is defined in \eqref{eq:gap-def}) and argue that it reduces at a rate $1/A_k,$ while the specifics of the gap estimate construction differ between algorithms and considered optimization settings. We carry out the main convergence argument in Section \ref{sec:gap-est-conv-analysis} and discuss choices of sampling probabilities and resulting complexity bounds in Section \ref{sec:sampling-distributions}. 

\subsection{Gap Estimate Construction and Main Convergence Result}\label{sec:gap-est-conv-analysis}

We start our convergence analysis with the following lemma, which defines the error sequence used in bounding $\Gap(\vxb_k, \vx).$ Observe that this lemma is generic in the sense that it does not depend on our choice of $\mFh_k$. We can view it as a constructive approach to bounding the gap function, which in turn motivates our choice of $\mFh_k$ as a means for controlling the error sequence $E_i,$ $i \geq 1,$ defined in the lemma. By convention (same as in Algorithm \ref{alg:main}), we take $a_0 = A_0 = 0.$

\begin{lemma}[A Generic Gap Bound]\label{lemma:gap-bnd}
    Consider algorithm updates specified in \eqref{eq:alg-iteration} and let $\vxb_k = \frac{1}{A_k}\sum_{i=1}^k a_i \vx_i$ for $k \geq 1$. Then, for any $k \geq 1$ and any $\vx \in \dom(g),$
    \begin{equation}\notag
        \Gap(\vxb_k, \vx) \leq \frac{\sum_{i=1}^k E_i + D(\vx, \vx_0) - \frac{A_k \gamma + 1}{2}\|\vx - \vx_k\|^2}{A_k},
    \end{equation}
    where
    \begin{equation}\notag
        E_i := a_i\big(\innp{\mF(\vx_i) - \mFh_{i}, \vx_i - \vx} - \frac{A_{i-1}\gamma + 1}{2}\|\vx_i - \vx_{i-1}\|^2,\quad i \geq 1.
    \end{equation}
\end{lemma}
\begin{proof}
    By monotonicity of $\mF$ and  convexity of $g,$ we have
    \begin{equation}\label{eq:gap-1}
        \Gap(\vxb_k, \vx) \leq \frac{1}{A_k}\sum_{i=1}^k a_i\big(\innp{\mF(\vx_i), \vx_i - \vx} + g(\vx_i) - g(\vx)\big).
    \end{equation}
    Adding and subtracting $\frac{1}{A_k}\sum_{i=1}^k a_i \innp{\mFh_i, \vx_i - \vx} - \frac{1}{A_k}D(\vx, \vx_0)$ in the last inequality, we now have
    \begin{equation}\label{eq:gap-2}
    \begin{aligned}
        \Gap(\vxb_k, \vx) \leq\; & \frac{1}{A_k}\sum_{i=1}^k a_i\big(\innp{\mFh_i, \vx_i - \vx} + g(\vx_i) - g(\vx)\big) - \frac{1}{A_k}D(\vx, \vx_0)\\
        &+ \frac{1}{A_k}\sum_{i=1}^k a_i\big(\innp{\mF(\vx_i) - \mFh_i, \vx_i - \vx} + \frac{1}{A_k}D(\vx, \vx_0).
    \end{aligned}
    \end{equation}
    Define $\phi_k(\vx) := -\sum_{i=1}^k a_i\big(\innp{\mFh_i, \vx_i - \vx} + g(\vx_i) - g(\vx)\big) +  D(\vx, \vx_0)$. This function is $(A_k \gamma + 1)$-strongly convex, as the Bregman divergence term $D(\vx, \vx_0)$ is 1-strongly convex and $g$ is $\gamma$-strongly convex. Looking back at \eqref{eq:alg-iteration}, we also observe that $\phi_k(\vx)$ is minimized by $\vx_k.$ Hence, we can conclude that
    \begin{equation}\label{eq:est-seq-1}
    \begin{aligned}
        \phi_k(\vx) &\geq \phi_k(\vx_k) + \frac{A_k \gamma + 1}{2}\|\vx - \vx_k\|^2\\
        &= \phi_{k-1}(\vx_k) + \frac{A_k \gamma + 1}{2}\|\vx - \vx_k\|^2.
    \end{aligned}
    \end{equation}
    Observing again that $\phi_{k-1}$ is $(A_{k-1}\gamma + 1)$-strongly convex and minimized by $\vx_{k-1},$ we get 
    $$
    \phi_{k-1}(\vx_k) \geq \phi_{k-1}(\vx_{k-1}) + \frac{A_{k-1}\gamma + 1}{2}\|\vx_k - \vx_{k-1}\|^2 = \phi_{k-2}(\vx_{k-1}) + \frac{A_{k-1}\gamma + 1}{2}\|\vx_k - \vx_{k-1}\|^2.
    $$
    Combining with \eqref{eq:est-seq-1} and unrolling the recursion, we have
    \begin{equation}\notag
        \phi_k(\vx) \geq \frac{A_k \gamma + 1}{2}\|\vx - \vx_k\|^2 + \sum_{i=1}^k \frac{A_{i-1} \gamma + 1}{2}\|\vx_i - \vx_{i-1}\|^2.
    \end{equation}
    Plugging the last inequality back into \eqref{eq:gap-2}, we now have
    \begin{equation}\label{eq:gap-3}
        \begin{aligned}
            \Gap(\vxb_k, \vx) \leq \; & -\frac{A_k \gamma + 1}{2A_k}\|\vx - \vx_k\|^2 - \sum_{i=1}^k \frac{A_{i-1} \gamma + 1}{2A_k}\|\vx_i - \vx_{i-1}\|^2\\
            &+ \frac{1}{A_k}\sum_{i=1}^k a_i\big(\innp{\mF(\vx_i) - \mFh_i, \vx_i - \vx} + \frac{1}{A_k}D(\vx, \vx_0).
        \end{aligned}
    \end{equation}
    To complete the proof, it remains to combine \eqref{eq:gap-3} with the definition of $E_i,$ and rearrange. 
\end{proof}

The rest of the convergence proof is algorithm-specific and is carried out by controlling the error sequence to ensure each error term $E_i$ only consists of terms that are either non-positive or telescoping, in expectation. This is the crux of the convergence analysis, specific to our work. We first  prove the following auxiliary lemma, which can be interpreted as a recursive variance bound for the stochastic estimates $\mFt_k$.
\begin{lemma}\label{claim:rec-var}
    For any $i \geq 1$ and any $j \in \{1, 2, \dots, m\},$ we have
\begin{equation}\label{eq:variance-bound}
    \E\big[\|\mF_j(\vx_{i}) - \mFt_{i-1, j}\|_*^2 \big] \leq \frac{5}{q_j}\sum_{i'=1}^i (1 - q_j/2)^{i - i'} \E\big[\|\mF_j(\vx_{i'}) - \mF_j(\vx_{i'-1})\|_*^2\big]. 
\end{equation}
\end{lemma}
    \begin{proof}
    By Young's inequality, we have that for any $\beta > 0,$
    \begin{equation}\label{eq:rec-young}
        \|\mF_j(\vx_{i}) - \mFt_{i-1, j}\|_*^2 \leq \Big(1 + \frac{1}{\beta}\Big)\|\mF_j(\vx_{i}) - \mF_j(\vx_{i-1})\|_*^2 + (1 + \beta)\|\mF_j(\vx_{i-1}) - \mFt_{i-1, j}\|_*^2.
    \end{equation}
    Let $\cF_{i+}$ denote the natural filtration, containing all randomness in the algorithm up to, but excluding the random choice of $j_i'$ ($\cF_{i+}$ however inludes the random choice of $j_i$). 
    By the tower property of expectation,
        \begin{equation}\label{eq:rec-tower}
            \E\big[\|\mF_j(\vx_{i}) - \mFt_{i, j}\|_*^2 \big] = \E\big[\E\big[\|\mF_j(\vx_{i}) - \mFt_{i, j}\|_*^2 | \cF_{i+} \big]\big].
        \end{equation}
     Observe that conditioned on $\cF_{i+}$, the only source of randomness is the random choice of $j_i'.$ With probability $q_j$, we have that $j_i' = j,$ in which case $\mFt_{i, j} = \mF_j(\vx_{i})$, while with the remaining probability,  $\mFt_{i, j} = \mFt_{i-1, j}$. Plugging into \eqref{eq:rec-tower} and simplifying, we get
     \begin{equation}\label{eq:rec-conditioning}
         \E\big[\|\mF_j(\vx_{i}) - \mFt_{i, j}\|_*^2 \big] = (1 - q_j)\E\big[\|\mF_j(\vx_{i}) - \mFt_{i-1, j}\|_*^2\big]. 
     \end{equation}
     Combining \eqref{eq:rec-young} and \eqref{eq:rec-conditioning}, we reach the following recursive inequality:
     \begin{equation}\notag
         \begin{aligned}
             \E\big[\|\mF_j(\vx_{i}) - \mFt_{i-1,j}\|_*^2 \big] \leq \; & \Big(1 + \frac{1}{\beta}\Big)\E\big[\|\mF_j(\vx_{i}) - \mF_j(\vx_{i-1})\|_*^2\big]\\
             &+ (1 + \beta)(1-q_j)\E\big[\|\mF_j(\vx_{i-1}) - \mFt_{i-2, j}\|_*^2\big]. 
         \end{aligned}
     \end{equation}
     In particular, choosing $\beta = q_j/4$ and recalling that it must be $q_j \in (0, 1),$ we get
     \begin{equation}\label{eq:recursive-variance}
         \begin{aligned}
             \E\big[\|\mF_j(\vx_{i}) - \mFt_{i-1, j}\|_*^2 \big] \leq \; & \frac{5}{q_j}\E\big[\|\mF_j(\vx_{i}) - \mF_j(\vx_{i-1})\|_*^2\big] + (1-q_j/2)\E\big[\|\mF_j(\vx_{i-1}) - \mFt_{i-2, j}\|_*^2\big]. 
         \end{aligned}
     \end{equation}
     To complete the proof, it remains to apply \eqref{eq:recursive-variance} recursively and recal that $\mFt_0 = \mF(\vx_0)$.  
    \end{proof}

In the next lemma, we show how to bound the individual error terms $E_i,$ which, combined with Lemma~\ref{lemma:gap-bnd} will lead to the desired convergence result.

\begin{lemma}\label{lemma:error-bnd}
    Let $E_i$ be defined as in Lemma~\ref{lemma:gap-bnd}. Then, for any $i \geq 1$ and any $\vx$ (possibly dependent on the algorithm randomness), we have
    \begin{equation}\notag
    \begin{aligned}
       \E [E_i] \leq\; & \E\big[a_i \innp{\mF(\vx_i) - \mFt_{i-1}, \vx_i - \vx} - a_{i-1}\innp{\mF(\vx_{i-1}) - \mFt_{i-2}, \vx_{i-1} - \vx}\big]\\
       &+ \frac{5{a_{i-1}}^2}{A_{i-1} \gamma + 1}\sum_{j=1}^m\sum_{i'=1}^{i-1} \frac{(1 - \frac{q_j}{2})^{i -1 - i'}}{{p_{j}}q_j} \E\big[\|\mF_j(\vx_{i'}) - \mF_j(\vx_{i'-1})\|_*^2\big]\\
       &- \frac{A_{i-1} \gamma + 1}{4}\E[\|\vx_i - \vx_{i-1}\|^2]
       \\
       &+ \E\Big[a_{i-1}\Big\langle{\frac{1}{p_{j_i}}(\mF_{j_i}(\vx_{i-1}) - \mFt_{i-2, j_i})-(\mF(\vx_{i-1}) - \mFt_{i-2}), \vx}\Big\rangle\Big].
    \end{aligned}
    \end{equation}
    where the expectation is taken w.r.t.\ all the randomness in the algorithm.
\end{lemma}
\begin{proof}
    Recalling the definitions of $E_i$ and $\mFh_i,$ we have
    \begin{equation}\label{eq:error-1}
        \begin{aligned}
            E_i =\; & a_i \innp{\mF(\vx_i) - \mFt_{i-1}, \vx_i - \vx} - \frac{a_{i-1}}{p_{j_i}}\innp{\mF_{j_i}(\vx_{i-1}) - \mFt_{i-2, j_i}, \vx_{i-1} - \vx}\\
            &+ \frac{a_{i-1}}{p_{j_i}} \innp{\mF_{j_i}(\vx_{i-1}) - \mFt_{i-2, j_i}, \vx_{i-1} - \vx_i} - \frac{A_{i-1}\gamma + 1}{2}\|\vx_i - \vx_{i-1}\|^2.
        \end{aligned}
    \end{equation}
    By Young's inequality, for any $\alpha_i > 0,$
    \begin{equation}\notag
        \frac{a_{i-1}}{p_{j_i}} \innp{\mF_{j_i}(\vx_{i-1}) - \mFt_{i-2, j_i}, \vx_{i-1} - \vx_i} \leq \frac{{a_{i-1}}^2}{2{p_{j_i}}^2 \alpha_i}\|\mF_{j_i}(\vx_{i-1}) - \mFt_{i-2, j_i}\|_*^2 + \frac{\alpha_i}{2}\|\vx_{i-1} - \vx_i\|^2. 
    \end{equation}
    Hence, choosing $\alpha_i = \frac{A_{i-1} \gamma + 1}{2}$, combining with \eqref{eq:error-1}, and adding and subtracing $a_{i-1}\innp{\mF(\vx_{i-1}) - \mFt_{i-2}, \vx_{i-1} - \vx}$, we have
    \begin{equation}\label{eq:error-2}
        \begin{aligned}
            E_i \leq \; & a_i \innp{\mF(\vx_i) - \mFt_{i-1}, \vx_i - \vx} - a_{i-1}\innp{\mF(\vx_{i-1}) - \mFt_{i-2}, \vx_{i-1} - \vx}\\
            &+ \frac{{a_{i-1}}^2}{{p_{j_i}}^2 (A_{i-1} \gamma + 1)}\|\mF_{j_i}(\vx_{i-1}) - \mFt_{i-2, j_i}\|_*^2
            - \frac{A_{i-1}\gamma + 1}{4}\|\vx_i - \vx_{i-1}\|^2\\
            &- a_{i-1} \Big\langle{\frac{1}{p_{j_i}}(\mF_{j_i}(\vx_{i-1}) - \mFt_{i-2, j_i}) - (\mF(\vx_{i-1}) - \mFt_{i-2}), \vx_{i-1} - \vx}\Big\rangle.
        \end{aligned}
    \end{equation}
    Let $\cF_i$ denote the natural filtration induced by the algorithm randomness up to, but not including, iteration $i.$ As is standard, we bound the expectation conditioned on $\cF_i$ and then use the tower property of expectation by which $\E[\cdot] = \E[\E[\cdot|\cF_i]].$ 
    Since $j_i$ is drawn independently of the history, we have that
    \begin{equation}\label{eq:in-prod-term-Ei}
    \begin{aligned}
         \E\Big[\frac{a_{i-1}}{p_{j_i}}\innp{\mF_{j_i}(\vx_{i-1}) - \mFt_{i-2, j_i}, \vx_{i-1}}\Big|\cF_i\Big]   
         &= a_{i-1}\innp{\mF(\vx_{i-1}) - \mFt_{i-2}, \vx_{i-1}}. 
    \end{aligned}
    \end{equation}
Thus, in the rest of the proof, we focus on bounding $\E[\frac{{a_{i-1}}^2}{2{p_{j_i}}^2 \alpha_i}\|\mF_{j_i}(\vx_{i-1}) - \mFt_{i-2, j_i}\|_*^2|\cF_i],$ where $\alpha_i = \frac{A_i \gamma + 1}{2}.$ Again, by independence of the random choice of $j_i$ from the history, we can write:
\begin{align}
    \E\Big[\frac{{a_{i-1}}^2}{2{p_{j_i}}^2 \alpha_i}\|\mF_{j_i}(\vx_{i-1}) - \mFt_{i-2, j_i}\|_*^2\Big] =\; & 
    \E\Big[\E\Big[\frac{{a_{i-1}}^2}{2{p_{j_i}}^2 \alpha_i}\|\mF_{j_i}(\vx_{i-1}) - \mFt_{i-2, j_i}\|_*^2|\cF_i\Big]\Big]\notag \\
    =\; &   \sum_{j=1}^m\frac{{a_{i-1}}^2}{2{p_{j}} \alpha_i}\E\big[\|\mF_{j}(\vx_{i-1}) - \mFt_{i-2, j}\|_*^2 \big]. \label{eq:error-3}
\end{align}
To complete the proof of the lemma, it remains to combine \eqref{eq:error-2}--\eqref{eq:error-3} with Lemma \ref{claim:rec-var}. 
\end{proof}

\begin{theorem}[Main Theorem]\label{thm:main}
    Consider iterates $\vx_k$ of Algorithm \ref{alg:main} and let $\vxb_k = \frac{1}{A_k}\sum_{i=1}^k a_i \vx_i$ for $k \geq 1$. Let $j_* = \argmin_{1 \leq j \leq m} q_j.$ If $\frac{{a_i}^2}{A_i \gamma + 1} \leq (1 + q_{j_*}/5)\frac{{a_{i-1}}^2}{A_{i-1} \gamma + 1}$ and $\frac{25 {\Lpq}^2 {a_{i-1}}^2}{A_{i-1} \gamma + 1} \leq  \frac{A_{i-2}\gamma + 1}{4}$ for all $i \geq 2$, then we have the following results after $k \leq K$ iterations of the algorithm. For any $\vx \in \dom(g)$ that is independent of the randomness of the algorithm, we have 
    \begin{equation}\notag
        \E[\Gap(\vxb_k, \vx)] \leq \frac{D(\vx, \vx_0) - \frac{A_k\gamma + 1}{4}\E[\|\vx - \vx_k\|^2]}{A_k},
    \end{equation}
    where the expectation is w.r.t.\ all the randomness in the algorithm (random choices $j_i, j_i'\,$ for $i \in \{1, \dots, K\}$). In particular, for $\gamma \geq 0$, if $\vx^*$ is a solution to \eqref{eq:main-problem}, then
    \begin{equation}\notag
       \E[\|\vx^* - \vx_k\|^2] \leq \frac{4 D(\vx^*, \vx_0)}{A_k \gamma + 1}.
    \end{equation}
    In the $\gamma = 0$ case, if $\frac{75 {\Lpq}^2 {a_i}^2}{2} - \frac{1}{4} \leq 0$ for $i \geq 1,$ then we further have 
    \begin{equation}\notag
        \E\Big[\sup_{\vx \in \dom(g)}\Gap(\vxb_k, \vx)\Big] \leq \frac{2 \sup_{\vy \in \dom(g)} D(\vy, \vx_0) }{A_k}. 
    \end{equation}
    All the conditions on the step sizes can be satisfied with $a_1 = A_1 = \sqrt{\frac{2}{3}}\frac{1}{10 \Lpq}$ and $A_k = A_1 \max\Big\{k, \, \big(1 + \min\big\{\frac{q_{j_*}}{11}, \frac{\gamma}{10 {\Lpq}}\}\big)^{k-1}\Big\}$. 
\end{theorem}
\begin{proof}
    From Lemma~\ref{lemma:error-bnd}, recalling that $a_0 = A_0 = 0,$ we have
    \begin{equation}\label{eq:sum-error}
    \begin{aligned}
        &\sum_{i=1}^k \E[E_i]\\
        \leq \; & \E[a_k \innp{\mF(\vx_k) - \mFt_{k-1}, \vx_k - \vx}\Big]\\
        &+\sum_{i=1}^{k-1}\bigg(\frac{5{a_{i}}^2}{A_{i} \gamma + 1}\sum_{j=1}^m\sum_{i'=1}^{i} \frac{(1 - \frac{q_j}{2})^{i - i'}}{{p_j}q_j} \E\big[\|\mF_{j}(\vx_{i'}) - \mF_{j}(\vx_{i'-1})\|_*^2\big] - \frac{A_{i} \gamma + 1}{4}\E[\|\vx_{i+1} - \vx_{i}\|^2]\bigg)\\
        &+ \sum_{i=1}^{k-1} \E\Big[a_{i}\Big\langle\frac{1}{p_{j_{i+1}}}(\mF_{j_{i+1}}(\vx_{i}) - \mFt_{i-1, j_{i+1}})-(\mF(\vx_{i}) - \mFt_{i-1}), \vx\Big\rangle \Big]. 
    \end{aligned}
    \end{equation}
    For the first term on the right-hand side of \eqref{eq:sum-error}, starting with the sum  decomposition of $\mF, \mFt$, using Young's inequality for each term, and applying Lemma \ref{claim:rec-var}, we have 
    \begin{align}
        &\; \E[a_k \innp{\mF(\vx_k) - \mFt_{k-1}, \vx_k - \vx}]\notag\\
        =\; & \sum_{j=1}^m \E[a_k \innp{\mF_{j}(\vx_k) - \mFt_{k-1, j}, \vx_k - \vx}]\notag\\
        \leq\; & \sum_{j=1}^m \E\Big[\frac{{a_k}^2}{p_j(A_k\gamma + 1)} \|\mF_{j}(\vx_k) - \mFt_{k-1, j}\|_*^2 + \frac{p_j(A_k \gamma + 1)}{4}\|\vx_k - \vx\|^2\Big]\notag\\
        =\; &  \sum_{j=1}^m \E\Big[\frac{{a_k}^2}{p_j(A_k\gamma + 1)} \|\mF_{j}(\vx_k) - \mFt_{k-1, j}\|_*^2\Big] + \E\Big[\frac{A_k \gamma + 1}{4}\|\vx_k - \vx\|^2\Big]\notag\\
        \leq \; & \E\bigg[\frac{5{a_k}^2}{A_k\gamma + 1} \sum_{j=1}^m \sum_{i'=1}^k \frac{(1 - \frac{q_j}{2})^{k - i'}}{p_j q_j} \E\big[\|\mF_{j}(\vx_{i'}) - \mF_{j}(\vx_{i'-1})\|_*^2\big] + \frac{A_k\gamma + 1}{4}\|\vx_k - \vx\|^2\bigg].\label{eq:first-term-bnd}
     \end{align}
    For the third term in \eqref{eq:sum-error}, observe that when $\vx = \vx^*$ (or any point independent from the algorithm randomness), this term is zero. Hence, we can write
    \begin{align}
       E_B:= & \; \sum_{i=1}^k \E\Big[a_{i-1}\Big\langle{\frac{1}{p_{j_i}}(\mF_{j_i}(\vx_{i-1}) - \mFt_{i-1, j_i})-(\mF(\vx_{i-1}) - \mFt_{i-1}), \vx}\Big\rangle \Big]\notag\\
        =&\; \sum_{i=1}^k \E\Big[a_{i-1}\Big\langle{\frac{1}{p_{j_i}}(\mF_{j_i}(\vx_{i-1}) - \mFt_{i-1, j_i})-(\mF(\vx_{i-1}) - \mFt_{i-1}), \vx - \vx_0}\Big\rangle \Big]\notag\\
        =&\; \E\bigg[\Big\langle\sum_{i=1}^k a_{i-1}\Big(\frac{1}{p_{j_i}}(\mF_{j_i}(\vx_{i-1}) - \mFt_{i-1, j_i})-(\mF(\vx_{i-1}) - \mFt_{i-1})\Big), \vx - \vx_0\Big\rangle \bigg]. \label{eq:sup-exp-to-exp-sup-term}
    \end{align}
    Observe further that terms $a_{i-1}\Big(\frac{1}{p_{j_i}}(\mF_{j_i}(\vx_{i-1}) - \mFt_{i-1, j_i})-(\mF(\vx_{i-1}) - \mFt_{i-1})\Big)$ are all mean-zero when conditioned on the filtration $\cF_i$ containing all randomness in the algorithm up to (but not including) iteration $i$. In particular, \eqref{eq:sup-exp-to-exp-sup-term} can be bounded using known arguments; see, e.g., \cite[Lemma 16]{alacaoglu2022stochastic} (see also Lemma \ref{lemma:variance-bregman} in Appendix \ref{appx:variance-bregman}, included for completeness), which when applied $E_B$ leads to
    \begin{align}
            E_B  \leq \; & \sum_{i=1}^k  \E\Big[\frac{{a_{i-1}}^2}{2}\big\|\frac{1}{p_{j_i}}(\mF_{j_i}(\vx_{i-1}) - \mFt_{i-1, j_i}) - (\mF(\vx_{i-1}) - \mFt_{i-1}) \big\|_*^2\Big] + \E\big[D(\vx, \vx_0)\big]\notag\\
            \leq \; & \sum_{i=1}^k  \E\Big[\frac{{a_{i-1}}^2}{2{p_{j_i}}^2}\big\|\mF_{j_i}(\vx_{i-1}) - \mFt_{i-1, j_i} \big\|_*^2\Big] + \E\big[D(\vx, \vx_0)\big]\notag\\
            = \; & \sum_{i=1}^k \sum_{j=1}^m \E\Big[\frac{{a_{i-1}}^2}{2{p_{j}}}\big\|\mF_{j}(\vx_{i-1}) - \mFt_{i-1, j} \big\|_*^2\Big] + \E\big[D(\vx, \vx_0)\big]\notag\\
            \leq\; & \sum_{i=1}^k \sum_{j=1}^m \frac{5{a_{i-1}}^2}{2}\sum_{i'=1}^{i-1}\frac{(1-q_j/2)^{i-1-i'}}{{p_j}q_j}\E\big[\big\|\mF_{j}(\vx_{i'}) - \mF_{j}(\vx_{i'-1}) \big\|_*^2\big] + \E\big[D(\vx, \vx_0)\big], \label{eq:EB-bound}
    \end{align}
    where the second inequality holds by the variance being bounded above by the second moment and the last inequality is by Lemma \ref{claim:rec-var}. 

    \noindent\textbf{Case 1: fixed $\vx.$} Consider first the case in which $\vx$ is independent of the algorithm randomness (in which case, as discussed above, $E_B = 0$). Combining \eqref{eq:sum-error} and \eqref{eq:first-term-bnd} with Lemma~\ref{lemma:gap-bnd}, we have
    \begin{equation}\label{eq:gap-almost-final}
        \begin{aligned}
            A_k \E[\Gap(\vxb_k, \vx)]
            \leq \; & D(\vx, \vx_0) - \frac{A_{k}\gamma + 1}{4}\E[\|\vx - \vx_k\|^2]\\
           & +\sum_{i=1}^{k}\frac{5{a_{i}}^2}{A_{i} \gamma + 1}\sum_{j=1}^m\sum_{i'=1}^{i} \frac{(1 - \frac{q_j}{2})^{i - i'}}{{p_j}q_j} \E\big[\|\mF_{j}(\vx_{i'}) - \mF_{j}(\vx_{i'-1})\|_*^2\big]\\
            &- \sum_{i=1}^k\frac{A_{i-1}\gamma + 1}{4}\E\big[\|\vx_i - \vx_{i-1}\|^2\big]. 
        \end{aligned}
    \end{equation}

Recalling that $a_0 = A_0 = 0$ and exchanging the order of summations, we can further simplify \eqref{eq:gap-almost-final} to
\begin{equation}\label{eq:gap-after-changed-sum-order}
    \begin{aligned}
    A_k \E[\Gap(\vxb_k, \vx)] \leq \; & D(\vx, \vx_0) - \frac{A_{k}\gamma + 1}{4}\E[\|\vx - \vx_k\|^2]\\
    &+\sum_{i'=1}^k \sum_{j=1}^m \frac{1}{p_j q_j}  \E\big[\|\mF_{j}(\vx_{i'}) - \mF_{j}(\vx_{i'-1})\|_*^2\big] \sum_{i=i'}^k \frac{5 {a_i}^2}{A_{i}\gamma + 1}\Big(1 - \frac{q_j}{2}\Big)^{i-i'}  \\
     &- \sum_{i=1}^k\frac{A_{i-1}\gamma + 1}{4}\E\big[\|\vx_i - \vx_{i-1}\|^2\big].
\end{aligned}
\end{equation}
By the theorem assumptions, $\frac{{a_i}^2}{A_{i}\gamma + 1}\leq \min_{1 \leq j \leq m} (1 + q_j/5)\frac{{a_{i-1}}^2}{A_{i-1}\gamma + 1}$ for all $i \geq 2,$ thus 
\begin{equation}\label{eq:aux-sum}
    \sum_{i=i'}^k \frac{5 {a_i}^2}{A_{i}\gamma + 1}\Big(1 - \frac{q_j}{2}\Big)^{i-i'} \leq \frac{25}{q_j}\frac{{a_{i'}}^2}{A_{i'}\gamma + 1}, 
\end{equation}
where we used $(1 - q_j/2)(1 + q_j/5) \leq (1-q_j/5),$ as $q_j \leq 1$ for all $j$, and $\sum_{i = i'}^k (1 - q_j/5)^{i - i'} \leq \frac{5}{q_j}.$ 

Combining \eqref{eq:aux-sum} with \eqref{eq:gap-after-changed-sum-order} and recalling the definition of $\Lpq$, we get
\begin{align*}
    A_k \E[\Gap(\vxb_k, \vx)] \leq \; & D(\vx, \vx_0) - \frac{A_{k}\gamma + 1}{4}\E[\|\vx - \vx_k\|^2]\\
     &+\sum_{i=1}^k\Big(\frac{25 {\Lpq}^2 {a_i}^2}{A_i \gamma + 1} - \frac{A_{i-1}\gamma + 1}{4}\Big)\E\big[\|\vx_i - \vx_{i-1}\|^2\big]\\
     \leq\; & D(\vx, \vx_0) - \frac{A_{k}\gamma + 1}{4}\E[\|\vx - \vx_k\|^2],
\end{align*}
as $\frac{25 {\Lpq}^2 {a_i}^2}{A_i \gamma + 1} - \frac{A_{i-1}\gamma + 1}{4} \leq 0$ by the theorem assumptions. In particular, for $\vx = \vx_*$, reordering the last inequality and recalling that $\Gap(\vxb_k, \vx_*) \geq 0,$ we get
\begin{equation}\notag
    \E[\|\vx_* - \vx_k\|^2] \leq \frac{4 D(\vx_*, \vx_0)}{A_{k}\gamma + 1},
\end{equation}
completing the proof of the first theorem claim.

\noindent\textbf{Case 2: possibly random $\vx.$} In this case, we focus on the setting with $\gamma = 0.$ Since $\vx$ is potentially random, we need to account for the error $E_B.$ Combining \eqref{eq:sum-error}, \eqref{eq:first-term-bnd}, and \eqref{eq:EB-bound} with Lemma~\ref{lemma:gap-bnd}, we get
    \begin{equation}\label{eq:gap-for-exp-sup}
        \begin{aligned}
            A_k \E[\Gap(\vxb_k, \vx)]
            \leq \; & 2\E[D(\vx, \vx_0)] - \frac{A_{k}\gamma + 1}{4}\E[\|\vx - \vx_k\|^2]\\
           & +\sum_{i=1}^{k}\frac{15{a_{i}}^2}{2}\sum_{j=1}^m\sum_{i'=1}^{i} \frac{(1 - \frac{q_j}{2})^{i - i'}}{{p_j}q_j} \E\big[\|\mF_{j}(\vx_{i'}) - \mF_{j}(\vx_{i'-1})\|^2\big]\\
            &- \sum_{i=1}^k\frac{1}{4}\E\big[\|\vx_i - \vx_{i-1}\|^2\big]. 
        \end{aligned}
    \end{equation}
    The terms in the last two lines sum up to a non-positive quantity, following the same argument as in Case 1 and using $\frac{75 {\Lpq}^2 {a_i}^2}{2} - \frac{1}{4} \leq 0$ from the theorem assumptions. As a result, rearranging the last inequality and using that $- \frac{A_{k}\gamma + 1}{4}\E[\|\vx - \vx_k\|^2] \leq 0,$ we get that for any $\vx \in \dom(g),$ 
    \begin{equation}\notag
        \E[\Gap(\vxb_k, \vx)]
            \leq \frac{2 \E[D(\vx, \vx_0)]}{A_k}.
    \end{equation} 
    To complete the proof of the second claim, it remains to apply this inequality for $$\vx \in \argsup_{\vy \in \dom(g)} \Gap(\vxb_k, \vy)$$ and use that $\sup_{\vx \in \dom(g)} D(\vx, \vx_0) \geq D(\vy, \vx_0)$ for any $\vy \in \dom(g)$. Note that $\argsup_{\vy \in \dom(g)} \Gap(\vxb_k, \vy)$ must be nonempty, since $\Gap(\vxb_k, \vy)$ is continuous for $\vy \in \dom(g)$ and, as explained in Section \ref{sec:problem-setup}, $g$ is assumed to have a compact domain.  

    \noindent\textbf{Growth of sequence $A_k$}. To complete the proof of the theorem, it remains to bound below the growth of sequence $A_k$ for $k \geq 1.$ Observe first that for all stated conditions on the step size it holds that if they are satisfied for $\gamma = 0,$ they are satisfied for all $\gamma \geq 0.$ Thus we first bound below the growth assuming $\gamma = 0.$ In this case, it is not hard to verify that $a_i = \frac{1}{\sqrt{150}\Lpq} = \sqrt{\frac{2}{3}}\frac{1}{10 \Lpq}$ satisfies all stated inequalities, and so we conclude that $A_k \geq \sqrt{\frac{2}{3}}\frac{k}{10 \Lpq},$ for all $k \geq 1.$

    Now consider the case where $\gamma > 0.$ We argue that in this case we can choose $\alpha > 0$ such that $A_k = A_1 (1 + \alpha)^{k-1}$ satisfies both required inequalities. Observe that in this case $a_k = A_k - A_{k-1} = A_1 \alpha (1+\alpha)^{k-2}$ for $k \geq 2,$ while $a_1 = A_1.$ Let $j_* = \argmin_{1 \leq j \leq m} q_j.$ The two inequalities that $a_k, A_k$ need to satisfy for $k \geq 2$ are
    \begin{equation}\label{eq:step-size-conditions}
        \frac{{a_k}^2}{A_k \gamma + 1} \leq (1 + q_{j_*}/5)\frac{{a_{k-1}}^2}{A_{k-1} \gamma + 1}\quad \text{and} \quad \frac{25 {\Lpq}^2 {a_{k-1}}^2}{A_{k-1} \gamma + 1} \leq  \frac{A_{k-2}\gamma + 1}{4}.
    \end{equation}
    The first condition in \eqref{eq:step-size-conditions} is equivalent to  $\frac{{a_k}^2}{{a_{k-1}}^2}  \leq (1 + q_{j_*}/5) \frac{A_{k} \gamma + 1}{A_{k-1} \gamma + 1},$ and since $A_k$ is increasing, it is satisfied for $\frac{{a_k}^2}{{a_{k-1}}^2}  \leq (1 + q_{j_*}/5).$ The last inequality is equivalent to $1 + \alpha \leq \sqrt{1 + q_{j_*}/5}.$ Since $q_{j_*} \in [0, 1],$ we have that $\sqrt{1 + q_{j_*}/5} \geq \sqrt{1 + 2q_{j_*}/11 + {q_{j_*}}^2/11^2} = 1 + q_{j_*}/11.$ Hence, $\alpha \leq q_{j_*}/11$ suffices to satisfy the first inequality in \eqref{eq:step-size-conditions}.

    For the second inequality in \eqref{eq:step-size-conditions}, observe first that the case $k = 2$ is satisfied for $a_1 = A_1 \leq \frac{1}{10 \Lpq}.$ For $k > 2,$ since $A_k$ is increasing and $A_{k}\gamma + 1 > A_k \gamma$ for all $k,$ we get that it suffices that $\frac{a_{k}}{A_{k-1}} \leq \frac{\gamma}{10 {\Lpq}}.$ Equivalently, this last inequality is $\alpha \leq \frac{\gamma}{10 {\Lpq}}$. Hence, we conclude that $A_k \geq \frac{1}{10 \Lpq} \big(1 + \min\big\{\frac{q_{j_*}}{11}, \frac{\gamma}{10 {\Lpq}}\}\big)^{k-1}.$ 
\end{proof}

A few remarks are in order here. First, the point $\vxb_k$ with respect to which the gap is bounded in Theorem~\ref{thm:main} for the case $\gamma = 0$ is not efficiently computable, since, as discussed in Section \ref{sec:algo}, the updates $\vx_k$ can be dense. Thus, the computation of $\vxb_k$ cannot be carried out efficiently in a block-wise manner, in general. In this case, we can follow a standard approach of outputing $\vx_{\hat{k}}$ for a randomly sampled index $\hat{k} \in \{1, 2, \dots, K\},$ according to the probability distribution defined by $\{\frac{a_1}{A_K}, \frac{a_2}{A_K}, \dots, \frac{a_K}{A_K}\},$ independent of randomness in the algorithm updates. For the case $\gamma =0$, $a_1 = a_2 = \dots = a_K$, so this distribution is uniform. A key insight here is that for any $\vx \in \dom(g),$
\begin{equation}\notag
    \E_{\hat{k}}[\Gap(\vx_{\hat{k}}, \vx)] = \sum_{k=1}^K \frac{a_k}{A_K }\Gap(\vx_{{k}}, \vx). 
\end{equation}
%
This tighter quantity can be bounded by our analysis if we simply omit the relaxation of it in \eqref{eq:gap-1}, at the beginning of our analysis. 
Of course, the variance of the estimate $\Gap(\vx_{\hat{k}}, \vx)$ can be lowered by sampling multiple indices instead of one and taking their empirical average. In particular, sampling order-$K/m$ such indices does not affect the algorithm's overall computational cost, modulo constant factors.  

We can conclude from Theorem \ref{thm:main} that the total number of iterations of Algorithm \ref{alg:main} to reach a solution with error (defined as the expectation of the supremum gap for $\gamma = 0$ and as the distance to solution for $\gamma > 0$) at most $\epsilon$ is 
\begin{equation}\label{eq:it-complexity-general}
    \cO\bigg(\min\bigg\{\frac{\Lpq \sup_{\vx \in \dom(g)} D(\vx, \vx_0)}{\epsilon},\; \Big(\frac{1}{q_{j_*}} + \frac{\Lpq}{\gamma}\Big)\log\Big(\frac{\Lpq D(\vx_*, \vx_0)}{\epsilon}\Big) \bigg\}\bigg).
\end{equation}
Note that the initialization is more expensive than individual iterations, requiring a full evaluation of $\mF(\vx_0).$ 

Because each iteration computes only two blocks of coordinates of $\mF$ and its cost is dominated by this computation, per-iteration cost of Algorithm \ref{alg:main} is generally much lower than the per-iteration cost of full-vector-update methods, especially when the number of blocks $m$ is large. For equally sized blocks and ``block coordinate-friendly'' problems, the cost of computing one block of coordinates of $\mF$ is of the order of $1/m$ times the cost of computing the full vector $\mF$, and thus the iterations of Algorithm \ref{alg:main} are $m$ times cheaper than iterations of full-vector-update methods.

Below we discuss concrete choices of the algorithm parameters (block sampling probability distributions $\vp$ and $\vq$) and how the resulting complexity bounds compare to those from related work.

\subsection{Sampling Distributions and Resulting Complexity}\label{sec:sampling-distributions}

We now discuss some basic choices of probability distributions $\vp, \vq$ and the resulting complexity bounds. As is usually the case for block coordinate methods, the improvements are primarily obtained when block-level Lipschitz constants are highly nonuniform and, moreover, nonuniform sampling is applied. 

\paragraph{Uniform Sampling.} If sampling is uniform, then each $p_j = q_j = 1/m,$ $j \in \{1, \dots, m\}.$ The parameter $\Lpq$ in this case becomes $\Lpq = m^{3/2}\sqrt{\sup_{\vx, \vy \in \dom(g), \vx \neq \vy} \frac{\sum_{j=1}^m \|\mF\bl{j}(\vx) - \mF\bl{j}(\vy)\|_{*}^2}{\|\vx - \vy\|^2}},$ which under the block-Lipschitzness assumption \eqref{eq:block-Lip-assmpt} becomes $\Lpq = m^{3/2}\|\vlambda\|_2,$ while the resulting iteration complexity is
\begin{equation}\label{eq:it-complexity-unif}
    \cO\bigg(\min\bigg\{\frac{m^{3/2}\|\vlambda\|_2 \sup_{\vx \in \dom(g)} D(\vx, \vx_0)}{\epsilon},\; \Big(m + \frac{m^{3/2}\|\vlambda\|_2}{\gamma}\Big)\log\Big(\frac{m\|\vlambda\|_2 D(\vx_*, \vx_0) }{\epsilon}\Big) \bigg\}\bigg).
\end{equation}
Prior work on randomized block coordinate methods for VI \cite{kotsalis2022simple} obtains a similar result, but with $m^{3/2}\|\vlambda\|_2$ replaced by $m^2 \|\vlambda\|_\infty.$ Since $\|\vlambda\|_\infty \leq \|\vlambda\|_2 \leq \sqrt{m}\|\vlambda\|_\infty$ with both inequalities being tight in general (the left inequality is tight when $\vlambda$ has one nonzero entry, the right inequality is tight when $\vlambda$ is uniform), \eqref{eq:it-complexity-general} provides a complexity guarantee that is never worse than the one in \cite{kotsalis2022simple} and improves upon it by a factor $\sqrt{m}$ when $\vlambda$ is highly non-uniform.

If we assume that the computation of a single block $\mF\bl{j}$ has cost proportional to $1/m$ of the cost of computing the full operator $\mF$ (a standard setting in which block coordinate methods are used), then compared to optimal full vector update (single block) methods \cite{korpelevich1976extragradient,nemirovski2004prox,nesterov2007dual,popov1980modification}, the overall complexity (runtime) resulting from \eqref{eq:it-complexity-unif} in the block coordinate setting is worse by a factor between $\sqrt{m}$ and $m,$ depending on whether $\vlambda$ is highly nonuniform or close to uniform. Compared to the complexity of the cyclic method \cite{song2023cyclic} under the block Lipschitz assumption \eqref{eq:block-Lip-assmpt}, \eqref{eq:it-complexity-unif} is worse by a factor $\sqrt{m}.$

In the context of variance reduction methods, compared to state of the art results \cite{alacaoglu2022stochastic,palaniappan2016stochastic,carmon2019variance}, the bound from \eqref{eq:it-complexity-unif} replaces $\sqrt{m}\|\vlambda\|_1$ with $m^{3/2}\|\vlambda\|_{2}$, which can be worse by a factor between $\sqrt{m}$ and $m.$ Thus, as will become apparent from the following discussion on importance-based sampling, non-uniform sampling is crucial for the usefulness of our results. 

\paragraph{Importance-Based Sampling.} Consider the following importance based sampling, where we would like to set $p_j, q_j \propto \sqrt{L_j}.$ A small technical consideration here is that it is undesirable to allow any $q_j$ to become too small compared to the uniform value $1/m,$ as the complexity bound for the strongly monotone case (the right part of the min in \eqref{eq:it-complexity-general}) grows with $\min_j 1/q_j.$ To deal with this issue, we can choose $q_j$ to be proportional to $\max\{\sqrt{L_j}, \frac{1}{m}\sum_{j'=1}^m \sqrt{L_{j'}}\}.$ We can see that in this case the normalizing constant can be bounded above by
\begin{equation}\label{eq:norm-const-q}
    \sum_{j=1}^m \max\{\sqrt{L_j}, \frac{1}{m}\sum_{j'=1}^m \sqrt{L_{j'}}\}\leq \sum_{j=1}^m\Big(\sqrt{L_j} + \frac{1}{m}\sum_{j'=1}^m \sqrt{L_{j'}}\Big) \leq 2 \sum_{j=1}^m \sqrt{L_j}. 
\end{equation}
Further, we can bound below the minimum value over sampling probabilities $q_j$ by
\begin{equation}\notag
    \min_{1\leq j \leq m} q_j =  \min_{1\leq j \leq m} \frac{\max\{\sqrt{L_j}, \frac{1}{m}\sum_{j'=1}^m \sqrt{L_{j'}}\}}{\sum_{l=1}^m \max\{\sqrt{L_l}, \frac{1}{m}\sum_{j'=1}^m \sqrt{L_{j'}}\}} \geq \frac{\frac{1}{m}\sum_{j'=1}^m \sqrt{L_{j'}}}{2\sum_{l=1}^m \sqrt{L_l}} = \frac{1}{2m},
\end{equation}
where in the last inequality we used $\max\{\sqrt{L_j}, \frac{1}{m}\sum_{j'=1}^m \sqrt{L_{j'}}\} \geq \frac{1}{m}\sum_{j'=1}^m \sqrt{L_{j'}}$ and \eqref{eq:norm-const-q}. 

Observe that even though we are now guaranteed that $\min_{1\leq j \leq m} q_j \geq \frac{1}{2m},$ this does not prevent possibly highly nonuniform sampling probabilities $\vq.$ In particular, in the extreme case where one of the Lipschitz parameters $L_j$ is much higher than the rest, $q_j$ can be as high as $1/2.$

The probability vector $\vp$ can be chosen either as $\vp = \vq$ or defined via $p_j = \frac{\sqrt{L_j}}{\sum_{j'=1}^m \sqrt{L_{j'}}},$ leading to qualitatively the same complexity bounds. In particular, under the block Lipschitz assumption \eqref{eq:block-Lip-assmpt}, we have $\|\mF_j(\vx) - \mF_j(\vy)\|_{*}^2 \leq {L_j}^2 \|\vx - \vy\|^2$ for all $j \in \{1, \dots, m\}$ and thus 
\begin{align*}
    \Lpq &:= \sqrt{\sup_{\vx, \vy \in \dom(g), \vx \neq \vy} \frac{\sum_{j=1}^m \frac{1}{p_j {q_j}^2}\|\mF_j(\vx) - \mF_j(\vy)\|_{*}^2}{\|\vx - \vy\|^2}}\\
    &\leq \sqrt{\sum_{j=1}^m p_j\frac{{L_j}^2}{{p_j}^2 {q_j}^2} } \leq \sqrt{\sum_{j=1}^m p_j{C_\vp}^2 {C_\vq}^2} = C_\vp C_\vq,
\end{align*}
where we used $p_j, q_j \geq \sqrt{L_j}$ and where $C_\vp$ and $C_\vq$ are the normalization constants for sampling probabilities $\vp, \vq,$ which are both $\cO\big(\sum_{j=1}^m \sqrt{L_{j}} = \|\vlambda\|_{1/2}^{1/2}\big),$ as discussed above. Thus, $\Lpq = \cO(\|\vlambda\|_{1/2}).$ 
Therefore, the bound on the number of iterations of our algorithm becomes
\begin{equation}\label{eq:it-complexity-nonuniform}
    \cO\bigg(\min\bigg\{\frac{\|\vlambda\|_{1/2} \sup_{\vx \in \dom(g)} D(\vx, \vx_0)}{\epsilon},\; \Big(m + \frac{\|\vlambda\|_{1/2}}{\gamma}\Big)\log\Big(\frac{\|\vlambda\|_{1/2} D(\vx_*, \vx_0)}{\epsilon}\Big) \bigg\}\bigg).
\end{equation}
In general, $ \|\vlambda\|_\infty \leq \|\vlambda\|_{1/2} \leq m^{3/2}\|\vlambda\|_2 \leq m^2 \|\vlambda\|_\infty.$ As a result, in block coordinate settings, complexity from \eqref{eq:it-complexity-nonuniform} is never worse than the complexity of the randomized block coordinate VI algorithm from \cite{kotsalis2022simple}. Further, under the block Lipschitz assumption \eqref{eq:block-Lip-assmpt}, our result in terms of the overall complexity (number of full computations of $\mF$) is never worse than the complexity of the cyclic block coordinate method from \cite{song2023cyclic} by a factor larger than $\sqrt{m},$ but it can be better by a factor $m$. Comparing to optimal full vector update methods \cite{korpelevich1976extragradient,nemirovski2004prox,nesterov2007dual,popov1980modification}, the iteration complexity \eqref{eq:it-complexity-nonuniform} can be higher by a factor $m^2$ when $\vlambda$ is close to uniform or of the same order if $\vlambda$ is highly nonuniform. Since the per-iteration cost can be $1/m$ lower, in the latter case our algorithm would be order-$m$ times faster than optimal full vector update methods.  

In the context of variance reduction methods, compared to state of the art results \cite{alacaoglu2022stochastic,palaniappan2016stochastic,carmon2019variance}, the bound from \eqref{eq:it-complexity-unif} replaces $\sqrt{m}\|\vlambda\|_1$ with $\|\vlambda\|_{1/2}$, which can be worse by a factor up to $\sqrt{m}$ (as $\vlambda$ becomes more uniform) or better by a factor up to $\sqrt{m}$ (as $\vlambda$ becomes more nonuniform). We recall once again that the latter case is when, in general, variance reduction-based methods improve over classical methods \cite{korpelevich1976extragradient,nemirovski2004prox,nesterov2007dual,popov1980modification}, thus we believe our result to be a useful complexity improvement for finite-sum settings.  
Interestingly, to the best of our knowledge, this is the first result for general finite-sum problems where the leading (dependent on Lipschitz parameter) term in the complexity does not explicitly depend on $m.$ 

\subsection{Combining Variance Reduction with Block Coordinate Strategies}

Because our method only requires that there is a sum-decomposition of the operator $\mF$ into $m$ component operators, it should be clear that our method can be used as both a finite-sum and block coordinate method simultaneously. In particular, if $\mF$ can be represented in the form
\begin{equation}\notag
    \mF(\vx) = \sum_{i=1}^n \mF_i(\vx),
\end{equation}
and, further, if $g$ is block coordinate separable over $\ell$ blocks, then we can consider the sum decoposition of $\mF$ into $m = n\ell$ blocks of the form
\begin{equation}\label{eq:vr-block}
    \mF(\vx) = \sum_{i=1}^n\sum_{j=1}^\ell \mU^j \mF_i(\vx).
\end{equation}
If $L_{ij}$ denotes the Lipschitz parameter of summand $i$ over block $j,$ then the vector $\vlambda$ would be formed by stacking all these Lipschitz parameters into a vector (of size $m = n\ell$) and the same bound on the number of iterations \eqref{eq:it-complexity-nonuniform} would apply. It is crucial here that our method supports lazy, sparse updates (see Algorithm \ref{alg:lazy} in Appendix \ref{appx:lazy-algo}) so that the resulting arithmetic complexity/runtime in this case becomes equal to the complexity of computing one full operator $\mF(\vx_0)$ plus the bound in \eqref{eq:it-complexity-nonuniform} times the complexity of computing a component $\mF_i\bl{j},$ which can be constant for small/constant block sizes. This should be contrasted with existing variance-reduced methods for GMVI \cite{cai2024variance,alacaoglu2022stochastic,palaniappan2016stochastic}, which do not take into account reducing per-iteration complexity for sparse component operators. The only exception in this part of the literature are methods specifically targetting bilinearly coupled primal-dual problems, such as \cite{alacaoglu2022complexity,song2022coordinate,carmon2020coordinate}. A comparison to relevant methods in the case of least absolute deviation, where component operators have a single nonzero coordinate, is provided in Section \ref{sec:apps-LAD}.


\section{Examples}\label{sec:apps}

In this section, we discuss different example applications of the results obtained in previous sections. Since nonuniform sampling is essential for the obtained bounds to be useful, whereas component Lipschitz parameters needed for selecting the component sampling distributions are primarily computable for linear operators, our focus is on important examples of GMVIs for which the operator $\mF$ is linear. 

\subsection{Policy Evaluation in Reinforcement Learning}

As our first example, we consider the variational inequality arising from policy evaluation in reinforcement learning, as introduced in \cite[Section 4]{kotsalis2022simpleII}. Here we consider the deterministic (or dynamic programming) version of the policy evaluation problem, where all problem parameters are known and fixed. 

More specifically, consider a Markov decision process (MDP) with a finite state and action space described by $(\cS, \cA, p, r, \beta),$ where $\cS = \{1, \dots, n\}$ denotes the state space, $\cA = \{1, \dots, N\}$ denotes the action space, $p: \cS\times \cS \times \cA \to [0, 1]$ denotes the transition probability function, $r: \cS\times \cS \times \cA \to \sR$ denotes the reward, and $\beta \in (0, 1)$ a discount factor. The dynamics of transitions of this MDP  is described as follows. When in a given state $s \in \cS,$ after taking action $a \in \cA,$ the process moves to a new state state $s^+ \in \cS$ with probability $p(s, s^+, a)$ and reaps the reward $r(s, s^+, a).$ A policy $\nu: \cS \times \cA \to [0, 1]$ specifies the probabilities of taking an action $a$ when in a state $s.$ Once a policy is fixed, the states follow a time-homogeneous Markov chain with transition probability matrix (kernel) $\mP$ with entries $\mP(s, s^+) = \sum_{a \in \cA} \nu(s, a) p(s, s^+, a).$ Following \cite{kotsalis2022simpleII}, we assume that this Markov chain has a unique stationary distribution $\vpi,$ which satisfies $\vpi = \mP \vpi.$

Policy evaluation refers to the problem of evaluating the value function $V_\nu$, defined by $$V_\nu(s) = \E\big[\sum_{t=0}^\infty \beta^t r_t|s_0 = s\big],$$ where $r_t = r(s_t, s_{t+1}, a_t)$ and $(s_t, s_{t+1}, a_t)$ refers to transitioning from state $s_t$ to state $s_{t+1}$ after taking action $a_t$ at time increment $t.$ When the policy function admits a linear parametric approximation $V_\vx = \mPhi \vx,$ where $\vx \in \sR^d$ is the parameter vector and $\mPhi$ a known feature matrix, \cite{kotsalis2022simpleII} shows that the policy evaluation can be reduced to a GMVI problem with $g \equiv 0$ and $\mG$ a linear operator defined by
\begin{equation}\label{eq:RL-op}
    \mG(\vx) = \mPhi^\top \mM(\mPhi \vx - \mR - \beta \mP \mPhi\vx),
\end{equation}
which is $\mu$-strongly monotone and $L$-Lipschitz for some $\mu > 0,$ $L > 0.$ In \eqref{eq:RL-op}, $\mM$ is a diagonal matrix with $\vpi$ on its main diagonal and $\mR$ is a reward vector defined based on the reward function $r$, probability function $p,$ and policy $\nu.$ 

To transfer strong monotonicity to the regularizing function $g,$ we can equivalently consider GMVI with $\mF(\vx) = \mG(\vx) - \mu \vx$ and $g(\vx) = \frac{\mu}{2}\|\vx\|_2^2.$ The two problems are equivalent since, due to the problem being unconstrained, the optimal solution satisfies $\mG(\vx_*) = \vzero,$ which is the same  as $\mF(\vx_*) + \nabla g(\vx_*) = \vzero.$ 

Following \cite{kotsalis2022simpleII}, denoting by $\Pi$ the stationary distribution of $(s, s^+, a)$, the operator $\mF$ can be written in a finite-sum form as
\begin{equation}\label{eq:RL-finite-sum}
    \mF(\vx) = \sum_{(s, s^+, a) \in \cS \times \cS \times \cA} \Pi(s, s^+, a) \big(\mPhi(s)(\innp{\mPhi(s) - \beta \mPhi(s^+), \vx} - r(s, s^+, a)) - \mu \vx\big),
\end{equation}
where $\mPhi(s)$ is the feature vector associated with state $s$ ($s^{\rm th}$ row of $\mPhi$). 

Denote the summands in \eqref{eq:RL-finite-sum} as $\mF_{s, s^+, a}(\vx)$. There are $\nnz(\Pi)$---the number of nonzero entries in $\Pi$---such nonzero summands. Note that the number of triples $(s, s^+, a)$ that occur with nonzero probability is typically much smaller than $|\cS|^2 |\cA| = n^2 N.$ It is not hard to see that the Lipschitz constant of each $(s, s^+, a)$ summand is $L_{s, s^+, a} = \Pi(s, s^+, a)(\|\mPhi(s)\|_2 \|\mPhi(s) - \beta\mPhi(s^+)\|_2 -\mu)$. Letting $\vlambda$ be the vector comprised of elements $L_{s, s^+, a},$ we get that given $\epsilon > 0,$ Algorithm \ref{alg:main} with the $\ell_2$ norm and $D(\vx, \vxh) = \frac{1}{2}\|\vx - \vxh\|_2^2$ as the Bregman divergence can output a point $\vx \in \sR^d$ such that $\|\vx - \vx_*\|_2 \leq \epsilon$ using
\begin{equation}\notag
    \cO\bigg(\Big(\nnz(\Pi) + \frac{d \|\vlambda\|_{1/2}}{\mu}\Big)\log\Big(\frac{ \|\vlambda\|_{1/2}\|\vx_0 - \vx_*\|_2}{\epsilon}\Big)\bigg)
\end{equation}
arithmetic operations, as each iteration uses $\cO(d)$ arithemetic operations and there are $\cO(\nnz(\Pi))$ (nonzero) components in the finite-sum decomposition of $\mF.$

In the regime where $\frac{d \|\vlambda\|_{1/2}}{\mu} = \cO(\nnz(\Pi)),$ this arithmetic complexity is near-linear in $\nnz(\Pi)$---the problem size. By comparison, the number of arithmetic operations taken by full vector update methods like \cite{korpelevich1976extragradient,nemirovski2004prox,nesterov2007dual,popov1980modification} would be $\cO\Big(\nnz(\Pi)\frac{L}{\mu}\log\big(\frac{ L \|\vx_0 - \vx_*\|_2}{\epsilon}\big)\Big)$, which can be much higher, depending on the values of $\|\vlambda\|_{1/2}$ (which can be comparable to $L$), $L,$ and $\mu.$

We remark here that we did not discuss the setting in which $\Pi$ is not known but can instead be sampled from---the setting originally considered in \cite{kotsalis2022simpleII}. The reason is that we did not consider here stochastic approximation settings for our method. Further, as our results crucially rely on importance-based sampling to lead to improved complexity bounds, it is unclear how to generalize such results to the stochastic approximation (infinite-sum) settings. However, there are at least two possible avenues for addressing such settings (in a useful way) starting from our results. First, and the more direct one, is to draw a sufficiently large number of samples from $\Pi$ and run the algorithm on the empirical version of the problem, after establishing appropriate concentration/uniform convergence results akin to e.g., \cite{shalev2010learnability}. Another possible approach would be to generalize the block coordinate version of our method to settings with stochastic estimates of the (block coordinate) operator and potentially further improve the complexity reported in \cite{kotsalis2022simpleII}, with the improvements coming from the importance-based sampling of (blocks of) coordinates. This seems possible, at the expense of added technical work, which we choose to omit to maintain focus and readability. 

\subsection{Matrix Games}

One of the most basic examples of GMVI are standard (simplex-simplex) matrix games, defined by
\begin{equation}\label{eq:matrix-games}
    \min_{\vz \in \Delta^d}\max_{\vy \in \Delta^n}\innp{\mA \vz, \vy},
\end{equation}
where $\Delta^d, \Delta^n$ denote the probability simplexes of size $d, n,$ respectively.

As a GMVI problem, \eqref{eq:matrix-games} can be posed in a standard way by stacking the primal and dual variables $\vx = (\vz^\top, \vy^\top)^\top$ and setting $\mF(\vx) = (\vy^\top \mA, - (\mA \vz)^\top)^\top,$ $g(\vx) = g_1(\vz) + g_2(\vy),$ with $g_1(\vz)$  being the indicator function of the probability simplex $\Delta^d$ and $g_2(\vy)$ the indicator function of  $\Delta^n$. We consider applying Algorithm \ref{alg:main} with $D(\vx, \vxh) = D_1(\vz, \vzh) + D_2(\vy, \vyh)$, where $D_1$ and $D_2$ are the Bregman divergences of the negative entropy function (a.k.a.\ the Kullback-Leibler divergences), which are 1-strongly convex w.r.t.\ the $\ell_1$ norm. We define the norm for the stacked vector $\vx = (\vz^\top, \vy^\top)^\top$ by $\|\vx\| = \sqrt{\|\vz\|_1^2 + \|\vy\|_1^2}.$   

There are different ways of writing the operator $\mF$ in a finite-sum form. An important consideration is that the probability simplex is not a (block or coordinate) separable set (except for the trivial partition containing one set with all the coordinates). Thus, even if $\mF$ is chosen as a sparse vector, updating all (primal $\vz$ or dual $\vy$) coordinates may still be unavoidable in each iteration. 

Let us consider the following sum-decomposition of $\mF$ into $m = n + d$ components, where we define
\begin{equation}\notag
    \mF_j(\vx) = \begin{cases}
        \begin{bmatrix}
            \mA_{j:} y_j \\
            \vzero_n
        \end{bmatrix}, &\text{ if } j \in \{1, \dots, n\}\\
        \begin{bmatrix}
            \vzero_d \\
            -\mA_{:j-n}z_{j-n}
        \end{bmatrix}, &\text{ if } j \in \{n + 1, \dots, n + d\}
    \end{cases},
\end{equation}
where $\mA_{j:}$ denotes the $j^{\rm th}$ row of matrix $\mA$, written as a column vector, $\mA_{:j}$ denotes the $j^{\rm th}$ column of $\mA$, and $\vzero_i$ for $i \in \{n, d\}$ denotes the column vector of size $i$ whose entries are all zeros.

To bound the parameter $\Lpq$ in this case, we set $p_j = q_j$ for all $j \in \{1, \dots, m\}$, and observe that for all $\vx, \vxh \in \Delta^d \times \Delta^n$:
\begin{align*}
    \sum_{j=1}^m \frac{1}{p_j {q_j}^2}\|\mF_j(\vx) - \mF_j(\vxh)\|_*^2 =&\; \sum_{j=1}^n \frac{1}{ {p_j}^3}\|\mA_{j:}\|_\infty^2(y_j - \hat{y}_j)^2 + \sum_{j'=1}^{d} \frac{1}{ {p_{j' + n}}^3}\|\mA_{:j'}\|_\infty^2(z_j - \hat{z}_j)^2.
\end{align*}
Define $\vrho \in \sR^n_+$ as the vector with entries $\rho_j = \|\mA_{j:}\|_\infty$ and $\vsigma \in \sR^d$ as the vector with entries $\sigma_j = \|\mA_{:j}\|_\infty.$ Set $p_j \propto {\rho_j}^{2/3}$ for $j \in \{1, \dots, n\}$ and $p_j \propto {\sigma_{j - n}}^{2/3}$ for $j \in \{n+1, \dots, n+d\}.$ Then:
\begin{align*}
    \sum_{j=1}^m \frac{1}{p_j {q_j}^2}\|\mF_j(\vx) - \mF_j(\vxh)\|_*^2 \leq &\; \|(\vrho^\top, \vsigma^\top)^\top\|_{2/3}^{2} \big(\|\vy - \vyh\|_2^2 + \|\vz - \vzh\|_2^2\big)\\
    &\leq \|(\vrho^\top, \vsigma^\top)^\top\|_{2/3}^{2} \|\vx - \vxh\|^2,
\end{align*}
where we used 
$\|\vy - \vyh\|_2^2 + \|\vz - \vzh\|_2^2 \leq \|\vy - \vyh\|_1^2 + \|\vz - \vzh\|_1^2  = \|\vx - \vxh\|^2$ (by the relationship between $\ell_p$ norms and the definition of $\|\cdot\|$). Thus, we get that $\Lpq \leq \|(\vrho^\top, \vsigma^\top)^\top\|_{2/3}$ and the runtime/arithmetic complexity of our algorithm becomes
\begin{equation}\notag
\begin{aligned}
    \cO\Big(\nnz(\mA) + \frac{\|(\vrho^\top, \vsigma^\top)^\top\|_{2/3}(n + d)(\log(n) + \log(d))}{\epsilon}\Big)
    = \widetilde{\cO}\Big(\nnz(\mA) + \frac{\|(\vrho^\top, \vsigma^\top)^\top\|_{2/3}(n + d)}{\epsilon}\Big),
\end{aligned}
\end{equation}
where $\nnz(\mA)$ comes from the initialization step in which $\mF(\vx_0)$ is computed, $n + d$ is the per-iteration cost, and $\ln(n) + \ln(d)$ is $\sup_{\vz \in \Delta_d, \vy \in \Delta_n}(D_1(\vz, \vz_0) + D_2(\vy, \vy_0))$ when the algorithm is initialized at $\vz_0 = \frac{1}{d}\vone_d,$ $\vy_0 = \frac{1}{n}\vone_n,$ where $\vone_{\cdot}$ is the vector of all ones of the appropriate size. 

Comparing to existing complexity results in the literature, our method can be faster in the settings where the matrix $\mA$ is dense and vectors $\vrho, \vsigma$ are highly nonuniform. In particular, optimal full vector update methods like \cite{nemirovski2004prox,nesterov2007dual} with the same choice of norms and the same Bregman divergence have runtime $\widetilde{\cO}\big(\nnz(\mA) \frac{\|\mA\|_\infty}{\epsilon}\big).$ Variance-reduced methods like \cite{carmon2020coordinate,alacaoglu2022stochastic} on the other hand lead to runtimes of the order $\cO\big(\nnz(\mA) + \frac{\sqrt{\nnz(\mA)(n+d)}(\|\vrho\|_\infty + \|\vsigma\|_\infty)}{\epsilon}\big)$ and $\cO\big(\nnz(\mA) + \frac{\sqrt{\nnz(\mA)}(\max\{\max_{1\leq i \leq n}\|\mA_{i:}\|_2, \max_{1\leq j \leq d}\|\mA_{:j}\|_2\})}{\epsilon}\big),$ which are lower than the runtimes of full vector update methods whenever $\mA$ is dense and/or column and vector $\ell_2$ norms are highly nonuniform. Our result further improves upon these results when $\mA$ is dense (so $\nnz(\mA) \leq nd$ can be much larger than $n + d$) and $\vrho, \vsigma$ are highly nonuniform (e.g., if $\|(\vrho^\top, \vsigma^\top)^\top\|_{2/3}$ is essentially dimension-independent). 

A few more remarks are in order here. Results in \cite{alacaoglu2022stochastic,carmon2020coordinate} rely on choosing sampling probabilities that are iteration-dependent. If the same were possible in our case, then we would choose $q_j = \frac{|y_j - \hat{y}_j|}{\|\vy - \vyh\|_1}$, $p_j \propto \|\mA_{j:}\|_\infty$ for $j \in \{1, \dots, n\}$, and $q_j = \frac{|z_{j-n} - \hat{z}_{j-n}|}{\|\vz - \vzh\|_1}$, $p_j \propto \|\mA_{:j-n}\|_\infty$ for $j \in \{n+1, \dots, n+d\}$ to further improve our result to scale with  $\|\vsigma\|_{1} + \|\vrho\|_{1}$ instead of $\|(\vrho^\top, \vsigma^\top)^\top\|_{2/3}.$ However, our current analysis does not support doing so, due to the way the ``recursive variance'' argument is carried out in Lemma~\ref{claim:rec-var}. It is an interesting question whether an alternative analysis could allow for iteration-dependent sampling probabilities $q_j$ leading to an overall improved complexity for this example. 

Second, we used $\cO(n + d)$ to bound the per-iteration cost; however, a slightly better bound applies. In particular, in iterations where $j \in \{1, \dots, n\}$ is selected, the per-iteration cost is $\cO(d),$ while when $j \in \{n + 1, \dots, n + d\}$ is selected, the per-iteration cost is $\cO(n).$ Based on the chosen sampling probabilities, the average per-iteration cost is $\cO\big(\frac{\|\vrho\|_{2/3}^{2/3}}{\|\vrho\|_{2/3}^{2/3} + \|\vsigma\|_{2/3}^{2/3}}d + \frac{\|\vsigma\|_{2/3}^{2/3}}{\|\vrho\|_{2/3}^{2/3} + \|\vsigma\|_{2/3}^{2/3}}n\big)$. Thus it is possible to argue that our expected runtime is of the order
\begin{equation}\notag
    \widetilde{\cO}\bigg(\nnz(\mA) + \frac{(\|(\vsigma^\top, \vrho^\top)^\top\|_{2/3}^{1/3})(d \|\vrho\|_{2/3}^{2/3} + n \|\vsigma\|_{2/3}^{2/3} )}{\epsilon}\bigg).
\end{equation}
Finally, it is possible to adapt our results to cases where only $\vrho$ or only $\vsigma$ is highly nonuniform. We only discuss the case where $\vrho$ is highly nonuniform, while the case where $\vsigma$ is highly nonuniform follows by a similar argument, using symmetry between the primal and the dual. Consider in this case the sum decomposition of $\mF$ with summands $\mF_j,$ $j \in \{1, \dots, n\},$ defined by
\begin{equation}\notag
    \mF_j(\vx) = \begin{bmatrix}
        \mA_{j:} y_j\\
        -({\mA_{j :}}^\top \vz) \ve_j
    \end{bmatrix},
\end{equation}
where $\ve_j$ denotes the $j^{\rm th}$ standard basis vector in $\sR^n.$ Then, setting $p_j = q_j \propto \|\mA_{j:}\|_{\infty}^{1/2}$, we get
\begin{align*}
    \sum_{j=1}^m \frac{1}{p_j {q_j}^2}\|\mF_j(\vx) - \mF_j(\vxh)\|_*^2 =&\; \sum_{j=1}^n \frac{1}{ {p_j}^3}\|\mA_{j:}\|_\infty^2(y_j - \hat{y}_j)^2 + \sum_{j'=1}^{n} \frac{1}{ {p_{j'}}^3}(\mA_{j' :}^\top (\vz - \vzh))^2\\
    \leq &\; \sum_{j=1}^n \frac{1}{ {p_j}^3}\|\mA_{j:}\|_\infty^2(y_j - \hat{y}_j)^2 + \sum_{j'=1}^{n} \frac{1}{ {p_{j'}}^3}\|\mA_{j' :}\|_\infty^2 \|\vz - \vzh\|_1^2\\
    = & \; \sum_{j=1}^n \frac{p_j}{ {p_j}^4}\|\mA_{j:}\|_\infty^2(y_j - \hat{y}_j)^2 + \sum_{j'=1}^{n} \frac{p_{j'}}{ {p_{j'}}^4}\|\mA_{j' :}\|_\infty^2 \|\vz - \vzh\|_1^2\\
    \leq & \; \|\vrho\|_{1/2}^2 \big(\|\vy - \vyh\|_2^2 + \|\vz - \vzh\|_1^2\big)\\
    \leq &\; \|\vrho\|_{1/2}^2 \|\vx - \vxh\|^2,
\end{align*}
where we used ${\mA_{j' :}}^\top (\vz - \vzh) \leq \|\mA_{j' :}\|_\infty \|\vz - \vzh\|_1$ (by H\"{o}lder's inequality), $p_j \leq 1,$ for all $j,$  $\|\vy - \vyh\|_2 \leq \|\vy - \vyh\|_1$, and the definitions of the sampling probabilities $p_j$ and the norm $\|\cdot\|.$ 

The resulting runtime in this case is thus
\begin{equation}
    \cO\Big(\nnz(\mA) + \frac{\|\vrho\|_{1/2}(n + d)(\log(n) + \log(d))}{\epsilon}\Big). 
\end{equation}

\subsection{Box-constrained $\ell_\infty$ Regression}

Box constrained regression problems with the $\ell_\infty$ norm are prevalent in theoretical computer science, particularly in the context of network flow problems \cite{sherman2017area,sidford2018coordinate} and linear programming \cite{lee2015efficient}. Such problems are defined by
\begin{equation}\label{eq:infty-reg}
    \min_{\vz \in [-1, 1]^d}\|\mA \vz - \vb\|_\infty.
\end{equation}
and have recently been studied formulated in a Lagrangian form as box-simplex matrix games, corresponding to: 
\begin{equation}\label{eq:box-simplex}
    \min_{\vz \in [-1, 1]^d}\max_{\vy \in \Delta^n}\innp{\mA \vz, \vy} - \innp{\vb, \vy}.
\end{equation}
Such problems are challenging in either the original \eqref{eq:infty-reg} or primal-dual \eqref{eq:box-simplex} form because the natural norm for the primal space is the $\ell_\infty$ norm, as it is the only $\ell_p$ norm for which the diameter of the feasible set (unit $\ell_\infty$ ball) is dimension-independent. To utilize algorithms working with non-Euclidean norms, it is typically needed that the distance generating function defining the selected Bregman divergence $D(\cdot, \cdot)$ is 1-strongly convex w.r.t.\ the selected problem norm. However, it is known that this is not possible for the $\ell_\infty$ norm (or any $\ell_p$ norm with $p > 2$ and $p$ not trivially close to 2) without the Bregman divergence scaling polynomially with the dimension; see, e.g., \cite{d2018optimal}. Obtaining algorithms with the optimal $1/\epsilon$ scaling, where $\epsilon > 0$ is the error parameter, while achieving an overall near-linear arithmetic complexity was a major open problem in theoretical computer science. A breakthrough result was obtained in \cite{sherman2017area}, using novel ``area convex'' functions to simultaneously regularize primal and dual variables.    

As a GMVI problem, \eqref{eq:box-simplex} can be posed by stacking the primal and dual variables $\vx = (\vx^\top, \vy^\top)^\top$ and setting $\mF(\vx) = (\vy^\top \mA, - (\mA \vz - \vb)^\top)^\top,$ $g(\vx) = g_1(\vz) + g_2(\vy),$ with $g_1(\vz)$  being the indicator function of the  $\ell_\infty$ ball $[-1, 1]^d$ and $g_2(\vy)$ the indicator function of the probability simplex $\Delta^n.$ For the sum decomposition of $\mF,$ we use $\mF_j$, for $j \in \{1, \dots, n\},$ defined by
\begin{equation}\label{eq:sum-decomp-dense-sparse}
    \mF_j(\vx) = \begin{bmatrix}
        (\mA_{:j}^\top \vy) \ve_j\\
        -(\mA_{:j} z_j - \vb)
    \end{bmatrix},
\end{equation}
where $\ve_j$ denotes the $j^{\rm th}$ standard basis vector in $\sR^d.$ Observe that because the primal space is coordinate-separable and we will choose below a coordinate separable distance generating function for the primal portion of the Bregman divergence, our algorithm can be seen as simultaneously performing a coordinate update on the primal and a full vector update on the dual. 

To apply our results, we use sampling probabilities $p_j = q_j$ specified below and define $\|\vz\|_\vp := \sqrt{\sum_{j=1}^d p_j {z_j}^2}$ for
the primal norm. The norm $\|\vx\|$ is then defined via $\|\vx\| = \sqrt{\|\vz\|_\vp^2 + \|\vy\|_1^2}.$ For the Bregman divergence $D(\vx, \vxh),$ we choose $D(\vx, \vxh) = D_1(\vz, \vzh) + D_2(\vy, \vyh)$ with $D_1(\vz, \vzh) = \frac{1}{2}\|\vz - \vzh\|_\vp^2$ and $D_2(\vy, \vyh)$ being the Bregman divergence of the negative entropy function. Observe that since $D_1$ is 1-strongly convex w.r.t.\ $\|\cdot\|_\vp$ and $D_2$ is 1-strongly convex w.r.t.\ the $\ell_1$ norm, we have that $D$ is 1-strongly convex w.r.t.\ $\|\cdot\|,$ as required. Observe further that $D_1(\vz, \vzh) \leq 2$ for any $\vz, \vzh \in [-1, 1]^d,$ since $\vz - \vzh \in [-2, 2]^d$ and $\sum_{j=1}^d p_j = 1,$ as $\vp$ is a probability vector. Additionally, $D(\vy, \vy_0) \leq \ln(n)$ for any $\vy$ and $\vy_0 = \frac{1}{n}\vone_n$.   

As before, to determine the complexity of our algorithm, we need to bound the parameter $\Lpq$ defined by \eqref{eq:Lpq-def}. To do so, recalling that we choose $p_j = q_j,$ we have
\begin{align*}
    \sum_{j=1}^m \frac{1}{p_j {q_j}^2}\|\mF_j(\vx) - \mF_j(\vxh)\|_*^2 \leq &\; \sum_{j=1}^d \frac{1}{ {p_j}^4}\|\mA_{:j}\|_\infty^2\|\vy - \vyh\|_1^2 + \sum_{j'=1}^{d} \frac{1}{ {p_{j'}}^3}\|\mA_{:j'}\|_\infty^2 (z_j - \hat{z}_j)^2\\
    =&\; \sum_{j=1}^d \frac{p_j}{ {p_j}^5}\|\mA_{:j}\|_\infty^2\|\vy - \vyh\|_1^2 + \sum_{j'=1}^{d} \frac{{p_{j'}}^2}{ {p_{j'}}^5}\|\mA_{:j'}\|_\infty^2 (z_j - \hat{z}_j)^2.
\end{align*}
Choosing $p_j \propto \|\mA_{:j}\|_\infty^{2/5}$ for $j \in \{1, \dots, d\}$ and defining $\vsigma$ as the vector of $\|\mA_{:j}\|_\infty$ as we did in the previous subsection, we get that
\begin{align*}
    \sum_{j=1}^m \frac{1}{p_j {q_j}^2}\|\mF_j(\vx) - \mF_j(\vxh)\|_*^2 \leq &\; \|\vsigma\|_{2/5}^2\Big(\|\vy - \vyh\|_1^2 + \sum_{j'=1}^{d}  {p_{j'}}^2 (z_j - \hat{z}_j)^2\Big)\\
    \leq &\; \|\vsigma\|_{2/5}^2\|\vx - \vxh\|^2,
\end{align*}
where the last inequality follows by the definition of the norm $\|\cdot\|$ and because $p_j \in (0, 1),$ $\forall j.$ 

Thus, $\Lpq \leq \|\vsigma\|_{2/5}.$ To bound the runtime/arithmetic complexity of our algorithm, we observe that $\nnz(\mA)$ operations are needed for initialization, while each iteration in which $j \in \{1, \dots, m\}$ is selected takes $\nnz(\mA_{:j})$ operations for computing $\mF_j$ and $\cO(n)$ operations to update $\vx$ (the update on the primal side takes a constant number of arithmetic operations, due to the coordinate separability discussed above). Thus, defining $c = \max_{1 \leq j \leq d} \nnz(\mA_{:j}),$ we get that the per-iteration cost is bounded by $c + n.$ The resulting runtime to output a point $\vxb_k$ with $\E[\Gap(\vxb_k, \vu)] \leq \epsilon$ for any $\vu \in [-1, 1]^d \times \Delta^n$ is 
\begin{equation}\notag
    \cO\Big(\nnz(\mA) + \frac{\|\vsigma\|_{2/5}(c + n)D(\vu, \vx_0)}{\epsilon}\Big). 
\end{equation}
Taking $\vu$ to be the point that maximizes $D(\vu, \vx_0)$ for an arbitrary $\vz_0 \in [-1, 1]^n$ and $\vy_0 = \frac{1}{n}\vone_n,$ we get that the runtime to obtain GMVI gap at most $\epsilon$ for the box-simplex problem \eqref{eq:box-simplex} is
\begin{equation}\label{eq:box-simplex-arithmetic-complexity}
    \cO\Big(\nnz(\mA) + \frac{\|\vsigma\|_{2/5}(c + n)\log(n)}{\epsilon}\Big). 
\end{equation}
On the other hand, to bound the optimality gap of the original $\ell_\infty$ regression problem \eqref{eq:infty-reg}, we observe that for $\vu = (\vz^\top, \vy^\top)^\top,$ the gap $\Gap(\vxb_k, \vu)$ simplifies to $\Gap(\vxb_k, \vu) = \innp{\mA \vzb_k - \vb, \vy} - \innp{\mA \vz - \vb, \vyb_k},$ where $\vxb_k = (\vzb_k^\top, \vyb_k^\top)^\top.$ Thus, defining $\vu$ by choosing $\vz = \vz_*,$ where $\vz_*$ solves \eqref{eq:infty-reg}, and $\vy = \sign((\mA\vzb_k - \vb)_{i_*}) \ve_{i_*}$ for $i_* = \argmax_{1 \leq i \leq n}(\mA\vzb_k - \vb)_{i},$ we get that $\|\mA\vzb_k - \vb\|_\infty - \|\mA \vz_* - \vb\|_\infty \leq \Gap(\vxb_k, \vu)$. Thus, it follows that our algorithm can output $\vzb_k$ with $\E[\|\mA\vzb_k - \vb\|_\infty - \|\mA \vz_* - \vb\|_\infty] \leq \epsilon$ using the number of arithmetic operations stated in \eqref{eq:box-simplex-arithmetic-complexity}. Because the gap in this case is bounded using a sparse $\vy$, our runtime in this case can be further improved by choosing $\vy_0$ as the vector of all zeros, using the standard (not weighted) Euclidean norm for the dual variables and selecting $p_j \propto \sqrt{\|\mA_{:j}\|_\infty},$ in which case the result improves the dependence on $\vsigma$ from $\|\vsigma\|_{2/5}$ to $\|\vsigma\|_{1/2}.$

Comparing to the state of the art, there are two main examples of similar runtimes. The original result by Sherman leads to the runtime of the order $\widetilde{\cO}\big(\nnz(\mA) \frac{\|\vsigma\|_\infty}{\epsilon}\big).$ Our result improves over this one for the original problem \eqref{eq:infty-reg} when $\|\vsigma\|_{1/2} = o\big(\frac{\nnz(\mA)}{n + c}\|\vsigma\|_\infty\big),$ which can happen when the matrix $\mA$ is dense and the vector $\vsigma$ is highly nonuniform. On the other hand, a coordinate method for the same problem was obtained in \cite{sidford2018coordinate}, with runtime $\widetilde{\cO}\big(d c + \frac{\min\{n, d\} + \sqrt{n \min\{d, \|\vz_*\|_2\}}c \|\vsigma\|_\infty}{\epsilon}\big).$ Our result can be seen to improve over this one whenever $\|\vsigma\|_{1/2} = o\big(\frac{\min\{n, d\} + \sqrt{n \min\{d, \|\vz_*\|_2\}}c \|\vsigma\|_\infty}{n + c}\big),$ which happens whenever $\vsigma$ is highly nonuniform, with the gap being larger depending on the size of the matrix. 

\subsection{Least Absolute Deviation}\label{sec:apps-LAD}

The least absolute deviation problem has as an input a matrix $\mA \in \sR^{n \times d}$ and a vector $\vb \in \sR^n$ and asks for finding a vector $\vz_* \in \sR^d$ that solves the minimization problem
\begin{equation}\label{eq:least-abs-dev}
    \min_{\vz \in \sR^d}\|\mA \vz - \vb\|_1.
\end{equation}
Equivalently, using standard Lagrangian duality, this problem can be written in a min-max form as 
\begin{equation}\label{eq:least-abs-dev-min-max}
    \min_{\vz \in \sR^d}\max_{\vy \in [-1, 1]^n}\innp{\mA \vz - \vb, \vy}.
\end{equation}
Consider a GMVI problem \eqref{eq:main-problem} formulated based on \eqref{eq:least-abs-dev-min-max}, using standard arguments, as follows. For the variables, we use the stacked vector $\vx = (\vz^\top, \vy^\top)^\top \in \sR^{d + n}$. The operator $\mF(\vx)$ is defined as $\mF(\vx) = (\vy^\top \mA, - (\mA \vz - \vb)^\top)^\top.$ The function $g(\vx)$ is independent of $\vz$ and is equal to the indicator of the set  $[-1, 1]^n$ (the unit $\ell_\infty$ ball). Let $\ve^d_j$ denote the $j$\textsuperscript{th} standard basis vector in $\sR^d$, and, similarly, let $\ve^n_i$ denote the $i$\textsuperscript{th} standard basis vector in $\sR^n$. For pairs of indices $i \in \{1, \dots, n\},$ $j \in \{1, \dots, d\}$ such that the entry $A_{ij} \neq 0,$ define the operator $\mF_{ij}$ by $\mF_{ij}(\vx) := (A_{ij}y_i \ve^d_j, -(A_{ij}z_j - b_i)\ve^n_i)^\top.$ Then
\begin{equation}\notag
    \mF(\vx) = \sum_{i, j: A_{ij} \neq 0} \mF_{ij}(\vx).
\end{equation}
The number of summation terms $m$ in the finite-sum form of $\mF$ stated above equals $\nnz(\mA)$---the number of nonzero entries of $\mA.$ 

Observe that in this case, the evaluation of the component operators $\mF_{ij}$ takes a constant number of arithmetic operations, while the iterations of our algorithm can be implemented with a constant number of arithmetic operations, using ideas and insights from Section \ref{sec:algo} (see also Algorithm \ref{alg:lazy} in Appendix~\ref{appx:lazy-algo}). Since the Lipschitz parameter of each $\mF_{ij}$ is $L_{ij} = |A_{ij}|,$ we get that our algorithm can be implemented to run in time 
\begin{equation}\notag
         \cO\bigg(\nnz(\mA) + \frac{\|\vect(\mA)\|_{1/2} \|\vx - \vx_0\|^2}{\epsilon}\bigg),
\end{equation}
to obtain $\E[\Gap(\vxb_k, \vx)] \leq \epsilon,$ 
where $\vect(\mA)$ denotes the vector obtained from stacking the columns of $\mA$ and $\vx$ is any vector such that its $\vy$-portion is in $[-1, 1]^n$, possibly dependent on the algorithm randomness. 

Specifically, for $\vx = (\vz_*^\top, \vy^\top)^{\top}$ with $y_i = \sign((\mA\vzb_k)_i - b_i)$ and $\vxb_k = (\vzb_k^\top, \vyb_k^\top)^\top$ we have 
\begin{equation}\notag
    \Gap(\vxb_k, \vx) = \|\mA\vzb_k - \vb\|_1 - \innp{\mA\vz_* - \vb, \vyb_k} \geq \|\mA\vzb_k - \vb\|_1 - \|\mA\vz_* - \vb\|_1,
\end{equation}
so the point output by the algorithm has expected optimality gap (for the original problem \eqref{eq:least-abs-dev}) at most $\epsilon.$

The total runtime of the algorithm can thus be bounded by
\begin{equation}\notag
         \cO\bigg(\nnz(\mA) + \frac{\|\vect(\mA)\|_{1/2} (n + \|\vz_* - \vz_0\|_2^2)}{\epsilon}\bigg).
\end{equation}
Compared to the state of the art methods, our algorithm has significantly lower complexity whenever $\vect(\mA)$ is highly nonuniform. In particular, full vector update methods such as \cite{korpelevich1976extragradient,nemirovski2004prox,nesterov2007dual,popov1980modification} and primal-dual methods such as \cite{chambolle2011first} lead to arithmetic complexity of the order $\cO(\nnz(\mA)\frac{\|\mA\|_2\|\vx - \vx_0\|_2^2}{\epsilon})$ for the min-max version of the problem \eqref{eq:least-abs-dev-min-max}, where $\|\mA\|_2$ is the operator (or spectral) norm of $\mA$; or, by a similar reduction to what is discussed above for our algorithm, $\cO(\nnz(\mA)\frac{\|\mA\|_2(n + \|\vz_* - \vz_0\|_2^2)}{\epsilon})$ for the original least absolute deviation problem stated in \eqref{eq:least-abs-dev}. This complexity is higher than ours whenever $\nnz(\mA) \|\mA\|_2 > \|\vect(\mA)\|_{1/2}.$ On the other hand, the same problem can be addressed with stochastic variants of the Chambolle-Pock method \cite{chambolle2011first}, such as those reported in \cite{song2022coordinate,alacaoglu2022complexity}, with runtimes $\cO(\nnz(\mA)\frac{\max_{1\leq j \leq d}\|\mA_{:j}\|_2\|\vx - \vx_0\|_2^2}{\epsilon})$ for the min-max version of the problem or $\cO(\nnz(\mA)\frac{\max_{1\leq j \leq d}\|\mA_{:j}\|_2(n + \|\vz_* - \vz_0\|_2^2)}{\epsilon})$ for the original least absolute deviation problem. Our result provides a lower runtime whenever $\max_{1\leq j \leq d}\|\mA_{:j}\|_2 \nnz(\mA) > \|\vect(\mA)\|_{1/2}.$

\section{Conclusion}

We obtained a new method that simultaneously addresses block coordinate and finite-sum settings of generalized variational inequalities of Minty type \eqref{eq:main-problem}. For the former settings, our method improves over state of the art and is the first method that transparently shows complexity improvements for block coordinate-type methods in settings of interest. In particular, when the associated block Lipschitz parameters are highly nonuniform, the improvement can be linear in the number of blocks (equal to the dimension for coordinate methods). In the finite-sum settings, our method can be seen as performing variance reduction of SAGA-type. Our result improves over state of the art whenever Lipschitz parameters of the components in the finite-sum are highly non-uniform, in which case the improvement is of the order $\sqrt{m},$ where $m$ is the number of summands in the finite-sum decomposition of the input operator. 

Some interesting open questions remain. While our method exhibits significant improvements over state of the art when the associated component Lipschitz parameters are highly nonuniform, in the other extreme---when the Lipschitz parameters are uniform---our method exhibits \emph{worse} complexity than either classical full vector update methods or state of the art variance-reduced methods. There is an interesting phenomenon related to complexity of finite-sum variational inequalities that arises in these results: going from classical methods to variance-reduced methods, the worst-case complexity can improve (with the extreme case being highly nonuniform Lipschitz parameters) or deteriorate (with the extreme case being uniform Lipschitz parameters) by a factor $\sqrt{m}$, where $m$ is the number of components in the finite-sum. Moving from the state of the art variance-reduced methods to our method, the worst-case complexity can improve or deteriorate by an additional factor $\sqrt{m}$. It is not known whether such trade-offs---sacrificing complexity in bad cases and gaining in good cases---can be avoided. An ``ideal'' result that one would hope for would be obtaining complexity scaling with $\|\vlambda\|_1$ instead of $\|\vlambda\|_{1/2}$ in our bounds, which would never be worse than the complexity of classical methods like \cite{korpelevich1976extragradient,nemirovski2004prox,nesterov2007dual,popov1980modification}. Alternatively, a bound scaling with $\|\vlambda\|_{2/3}$ would never be worse than the complexity of state of the art variance-reduced methods such as  \cite{palaniappan2016stochastic,alacaoglu2022stochastic}. Finally, ruling out such results and showing that the aforementioned trade-offs are necessary would be intriguing. 

\section*{Acknowledgements}

This research was supported in part by the Air Force Office of Scientific Research under award number FA9550-24-1-0076 and by the U.S.\ Office of Naval Research under contract number  N00014-22-1-2348. Any opinions, findings and conclusions or recommendations expressed in this material are those of the author(s) and do not necessarily reflect the views of the U.S. Department of Defense. 

The author thanks George Lan, who pointed out the gap in the literature on randomized block coordinate methods, which is how this paper came to be. The author also thanks Ahmet Alacaoglu, Yudong Chen, and George Lan for helpful discussions regarding applications of the obtained method \& complexity results. 

\bibliographystyle{abbrv}
\bibliography{references,references-vr}

\newpage
\appendix

\section{Implementation Version of the Algorithm}\label{appx:lazy-algo}

In this section, we provide the pseudocode for the implementation version of Algorithm \ref{alg:main}. Since the algorithm only maintains scalars and vectors of size $m$ and $d,$ the total memory requirement of the algorithm is $\cO(m+d) = \cO(d).$ The per-iteration cost of the algorithm is determined by the cost to compute one block $\mF\bl{j}$ of $\mF$ and update the coordinate blocks of $\vx$ needed for the computation of $\mF\bl{j}.$ In most cases of interest, the minimization problems defining the updates for $\vx\bl{j}$ are  either solvable in closed form or in (near-)linear time needed for reading $\vx\bl{j}.$ In such cases, the per-iteration cost is dominated by computing $\mF\bl{j}$ for $j \in\{j_1, j_2\}$ and agrees with the standard cost of block coordinate-style methods.

\begin{algorithm}
\caption{Randomized Extrapolated BCM (Lazy Implementation Version)}\label{alg:lazy}
    \begin{algorithmic}[1]
        \State \textbf{Input}: $\vx_0 \in \dom(g),\, \vp, \vq \in \Delta_m, K, \gamma, \Lpq$
        \State \textbf{Initialization}: $\mFt \leftarrow \mFt_{-} \leftarrow \mF(\vx_0),$ $a \leftarrow A \leftarrow 0,$ $\vz \leftarrow \vzero_d, \vx \leftarrow \vx_0$, $\mA_{\rm last} \leftarrow \vzero_m$, $\vxb \leftarrow \vzero$, $q_{j_*} = \min_{1\leq j \leq m}q_j$
        \State If $\gamma = 0,$ draw $\lceil K/m \rceil$ numbers from $\{1, \dots K\}$ uniformly at random and assign to the set $\cS_K$. Otherwise set $\cS_K = \emptyset$
        \For{$k = 1:K$}
        \State $a_{-} \leftarrow a$
        \State $a \leftarrow \begin{cases}
            \sqrt{\frac{2}{3}}\frac{1}{10 \Lpq}, &\text{ if } \gamma = 0\\
            \min\{\sqrt{1 + q_{j_*}/5}\, a_{-},\, \frac{A\gamma + 1}{10 \Lpq}\} , &\text{ if } \gamma \neq 0
        \end{cases},\;$  $A \leftarrow a + A$
        \State \parbox[t]{\dimexpr\textwidth-\leftmargin-\labelsep-\labelwidth}{Randomly draw $j_1, j_2$ from $\{1, \dots, m\}$ according to the probability distributions $\vp \in \Delta_m,$ $\vq \in \Delta_m,$ respectively, independently of the history and of each other \strut} 
        \State \parbox[t]{\dimexpr\textwidth-\leftmargin-\labelsep-\labelwidth}{For all $j \in \{1, \dots, m\}$ such that $\mF\bl{j_1}$ depends on $\vx\bl{j}$, update $\vz\bl{j} \leftarrow(A - \mA_{\rm last}[j])\mFt\bl{j},$ $\vx\bl{j} \leftarrow \argmin_{\vu\bl{j} \in \R^{|\cS^j|}}\big\{\innp{\vz\bl{j}, \vu\bl{j}} + A g^j(\vu\bl{j}) + D^j(\vu\bl{j}, \vx_0\bl{j})\big\}$ and compute $\mF\bl{j_1}$\strut}
        \State $\mA_{\rm last}[j_1] \leftarrow k$
        \State $\mFh\bl{j_1} \leftarrow \mFt\bl{j_1} + \frac{a_{-}}{a\, p_{j_1}}(\mF\bl{j_1}(\vx) - \mFt_{-}\bl{j_1})$
        \State $\vz\bl{j_1} \leftarrow \vz\bl{j_1} + a \mFh\bl{j_1}$
        \State $\vx\bl{j_1} \leftarrow \argmin_{\vu\bl{j_1} \in \R^{|\cS^{j_1}|}}\big\{\innp{\vz\bl{j_1}, \vu\bl{j_1}} + A g^{j_1}(\vu\bl{j_1}) + D^j(\vu\bl{j_1}, \vx_0\bl{j_1})\big\}$
        \State \parbox[t]{\dimexpr\textwidth-\leftmargin-\labelsep-\labelwidth}{For all $j \in \{1, \dots, m\}\backslash\{j_1\}$ such that $\mF\bl{j_2}$ depends on $\vx\bl{j}$, update $\vz\bl{j} \leftarrow(A - \mA_{\rm last}[j])\mFt\bl{j},$ $\vx\bl{j} \leftarrow \argmin_{\vu\bl{j} \in \R^{|\cS^j|}}\big\{\innp{\vz\bl{j}, \vu\bl{j}} + A g^j(\vu\bl{j}) + D^j(\vu\bl{j}, \vx_0\bl{j})\big\}$ and compute $\mF\bl{j_2}(\vx)$\strut}
        \State $\mFt\bl{j_2}_{-} \leftarrow \mFt\bl{j_2}$
        \State $\mFt\bl{j_2} \leftarrow \mF\bl{j_2}(\vx)$
        \If{$k \in \cS_K$}
        \State $\vxb \leftarrow \vxb + \frac{1}{|\cS_K|} \vx$
        \EndIf
        \EndFor
        \State \Return  $\vx, \vxb$
    \end{algorithmic}
\end{algorithm}

\section{Variance of Sum Inner Product for General Norms}\label{appx:variance-bregman}

In this section we provide a lemma that allows us to bound the expectation of the inner-product of a sum $\innp{\sum_{i=1}^m \vu_i, \vx}$, where $\vu_i$ are zero-mean and independent under conditioning on a filtration process, but $\vx$ is \emph{not} necessarily independent of $\vu_i.$ When working with Euclidean norms,  we can simply use Young's inequality to write $\innp{\sum_{i=1}^m \vu_i, \vx} \leq \frac{1}{2}\|\sum_{i=1}^m \vu_i\|_2^2 + \frac{1}{2}\|\vx\|_2^2$ and use conditional pairwise independence of $\vu_i$'s and the tower property of expectation to write $\E\big[\|\sum_{i=1}^m \vu_i\|_2^2\big] = \sum_{i=1}^m \E[\|\vu_i\|_2^2].$ This is not possible for more general norms, as they are not necessarilly inner product-induced. Instead, we rely upon the following lemma, adapted from \cite[Lemma 3.5]{alacaoglu2022stochastic} and stated here for completeness.

\begin{lemma}\cite[Lemma 16]{alacaoglu2022stochastic}\label{lemma:variance-bregman}
    Let $\cF = (\cF_k)_{k \geq 1}$ be a filtration and $\vu_k$ be a stochastic process adapted to $\cF$ with $\E[\vu_k|\cF_k] = 0.$ Given a fixed (independent of $\vu_k$, $k \geq 1$) $\vx_0 \in \dom(g)$ and any (possibly dependent on $\vu_k$, $k \geq 1$) $\vx \in \dom(g)$, for any $K \geq 1,$
    \begin{equation}
        \E\bigg[\sum_{k=1}^K \innp{\vu_k, \vx}\bigg] \leq \E[D(\vx, \vx_0)] + \frac{1}{2}\sum_{k=1}^K \E\big[\|\vu_k\|_2^2\big].
    \end{equation}
\end{lemma}

Lemma \ref{lemma:variance-bregman} is used in the proof of Theorem \ref{thm:main}, with $$\vu_k = a_{k-1}\Big(\frac{1}{p_{j_k}}(\mF_{j_k}(\vx_{k-1}) - \mFt_{k-1, j_k})-(\mF(\vx_{k-1}) - \mFt_{k-1})\Big),$$
which clearly satisfies the assumptions stated in the lemma. 

\end{document}

%% file: prologue.tex
\usepackage[margin=1in]{geometry}
 \usepackage{times}
 \usepackage{amsmath,amsfonts,amsthm,bm}

\usepackage{dsfont}
\usepackage{mathtools}
\usepackage{cite}

\usepackage[colorlinks,urlcolor=blue,citecolor=blue,linkcolor=blue]{hyperref}
 \usepackage{color}

\usepackage{thm-restate}

\usepackage{algorithm,algpseudocode}

\usepackage{graphicx}
\usepackage{subfigure}

\newtheorem{theorem}{Theorem}
\newtheorem{lemma}{Lemma}

\def\1{\bm{1}}

\def\vzero{{\bm{0}}}
\def\vone{{\bm{1}}}

\def\vlambda{{\bm{\lambda}}}
\def\vrho{{\bm{\rho}}}
\def\vsigma{{\bm{\sigma}}}
\def\vpi{{\bm{\pi}}}

\def\vb{{\bm{b}}}

\def\ve{{\bm{e}}}

\def\vp{{\bm{p}}}
\def\vq{{\bm{q}}}

\def\vu{{\bm{u}}}

\def\vx{{\bm{x}}}
\def\vy{{\bm{y}}}
\def\vz{{\bm{z}}}

\def\vxb{\bar{\bm{x}}}
\def\vzb{\bar{\bm{z}}}
\def\vyb{\bar{\bm{y}}}

\def\vxh{\hat{\bm{x}}}
\def\vyh{\hat{\bm{y}}}
\def\vzh{\hat{\bm{z}}}

\def\mA{{\bm{A}}}

\def\mF{{\bm{F}}}
\def\mG{{\bm{G}}}

\def\mI{{\bm{I}}}

\def\mM{{\bm{M}}}

\def\mP{{\bm{P}}}
\def\mQ{{\bm{Q}}}
\def\mR{{\bm{R}}}

\def\mU{{\bm{U}}}

\def\mPhi{{\bm{\Phi}}}

\def\mFh{\widehat{\bm{F}}}
\def\mFt{\widetilde{\bm{F}}}

\def\cA{{\mathcal{A}}}

\def\cF{{\mathcal{F}}}

\def\cO{{\mathcal{O}}}

\def\cS{{\mathcal{S}}}

\def\sR{{\mathbb{R}}}

\newcommand{\E}{\mathbb{E}}

\newcommand{\R}{\mathbb{R}}

\DeclareMathOperator*{\argmax}{arg\,max}
\DeclareMathOperator*{\argsup}{arg\,sup}
\DeclareMathOperator*{\argmin}{arg\,min}

\DeclareMathOperator{\sign}{sign}

\newcommand{\innp}[1]{\langle #1 \rangle}
\usepackage[colorinlistoftodos]{todonotes}
\newcommand{\Gap}{\mathrm{Gap}}

\newcommand{\dom}{\mathrm{dom}}
\newcommand{\vect}{\mathrm{vec}}
\newcommand{\nnz}{\mathrm{nnz}}

\newcommand{\Lpq}{L_\mathbf{p, q}}

\newcommand{\bl}[1]{^{(#1)}}